\documentclass{amsart}
\usepackage[utf8]{inputenc}
\usepackage{amsfonts}
\usepackage{amsmath}
\usepackage{amssymb}
\usepackage{amsthm}
\usepackage{caption}
\usepackage[noadjust]{cite}
\PassOptionsToPackage{hyphens}{url}
\usepackage[colorlinks=true,citecolor={blue},linkcolor = purple]{hyperref}
\usepackage{mathtools}
\usepackage{multirow}
\usepackage{hhline}
\usepackage{subcaption}
\usepackage[table,xcdraw,dvipsnames]{xcolor}
\usepackage{tikz}

\theoremstyle{definition}
\newtheorem{theorem}{Theorem}[section]
\newtheorem{lemma}[theorem]{Lemma}
\newtheorem{corollary}[theorem]{Corollary}
\newtheorem{proposition}[theorem]{Proposition}
\newtheorem{definition}[theorem]{Definition}

\theoremstyle{remark}
\newtheorem{example}[theorem]{Example}
\newtheorem{remark}[theorem]{Remark}

\counterwithin{figure}{section}
\counterwithin{table}{section}
\counterwithin{equation}{section}

\newcommand{\ZZ}{\mathbb{Z}}
\newcommand{\Fp}{{\mathbb{F}_p}}
\newcommand{\Fpbar}{{\overline{\mathbb{F}}_{p}}}
\newcommand{\gfb}{\mathcal{G}_\ell(\overline{\mathbb{F}}_p)}
\newcommand{\gf}{\mathcal{G}_\ell(\mathbb{F}_p)}
\NewDocumentCommand{\slp}{ O{\ell} O{p} }{\mathcal{S}_{#1}^{#2}}

\DeclarePairedDelimiter{\set}{\{}{\}}
\DeclarePairedDelimiter{\brac}{(}{)}

\begin{document}

\title[The Spine of a Supersingular  $\ell$-Isogeny graph]{The Spine of a Supersingular $\ell$-Isogeny graph}

\author{Taha Hedayat}
\address{Department of Mathematics and Statistics\\
	University of Calgary, 2500 University Drive NW\\
	Calgary, AB T2N 1N4, Canada}
\email{taha.hedayat@ucalgary.ca}

\author{Sarah Arpin}
\address{Virginia Tech Department of Mathematics\\ 225 Stanger Street \\Blacksburg, VA 24061, USA}
\email{sarpin@vt.edu}

\author{Renate Scheidler}
\address{Department of Mathematics and Statistics\\
	University of Calgary, 2500 University Drive NW\\
	Calgary, AB T2N 1N4, Canada}
\email{rscheidl@ucalgary.ca}

\begin{abstract}
   Supersingular elliptic curve $\ell$-isogeny graphs over finite fields offer a setting for a number of quantum-resistant cryptographic protocols. The security analysis of these schemes typically assumes that these graphs behave randomly. Motivated by this debatable assertion, we explore structural properties of these graphs. We detail the behavior, governed by congruence conditions on $p$, of the $\ell$-isogeny graph over $\Fp$ when passing to the spine, i.e.\ the subgraph induced by the $\Fp$-vertices in the full $\ell$-isogeny graph. We describe the diameter of the spine and offer numerical data on the number of vertices, over both $\Fp$ and $\Fpbar$, in the center of the $\ell$-isogeny graph. Our plots of these counts exhibit a wave-shaped pattern which supports the assertion that centers of supersingular $\ell$-isogeny graphs exhibit the same behavior as those of random $(\ell+1)$-regular graphs.
\end{abstract}

\keywords{Supersingular elliptic curve $\ell$-isogeny graph, spine, graph diameter, graph center}
\subjclass{14H52, 11G20, 11-04, 11-11, 05C40}

\maketitle

\begin{center}
\scalebox{.5}{
\begin{tikzpicture}

\draw[thick] (-2,0)--(0,2)--(2,0)--(0,-2)--(-2,0);
\draw[thick] (-2,0)--(1,1)--(0,-2)--(-1,1)--(2,0)--(-1,-1)--(0,2)--(1,-1)--(-2,0);
\draw[thick] (0,2)--(0,-2);
\draw[thick] (-2,0)--(2,0);
\draw[thick](-1,1)--(1,-1);
\draw[thick](1,1)--(-1,-1);

\draw[thick](-5,5)--(-5,-5);
\draw[thick](-5,5)--(-6,-5);
\draw[thick](-5,5)--(-7,-5);
\draw[thick](-5,5)--(-8,-5);
\draw[thick](-5,5)--(-9,-5);

\draw[thick](-6,5)--(-5,-5);
\draw[thick](-6,5)--(-6,-5);
\draw[thick](-6,5)--(-7,-5);
\draw[thick](-6,5)--(-8,-5);
\draw[thick](-6,5)--(-9,-5);

\draw[thick](-7,5)--(-5,-5);
\draw[thick](-7,5)--(-6,-5);
\draw[thick](-7,5)--(-7,-5);
\draw[thick](-7,5)--(-8,-5);
\draw[thick](-7,5)--(-9,-5);

\draw[thick](-8,5)--(-5,-5);
\draw[thick](-8,5)--(-6,-5);
\draw[thick](-8,5)--(-7,-5);
\draw[thick](-8,5)--(-8,-5);
\draw[thick](-8,5)--(-9,-5);

\draw[thick](-9,5)--(-5,-5);
\draw[thick](-9,5)--(-6,-5);
\draw[thick](-9,5)--(-7,-5);
\draw[thick](-9,5)--(-8,-5);
\draw[thick](-9,5)--(-9,-5);

\draw[thick](5,5)--(5,-5);
\draw[thick](5,5)--(6,-5);
\draw[thick](5,5)--(7,-5);
\draw[thick](5,5)--(8,-5);
\draw[thick](5,5)--(9,-5);

\draw[thick](6,5)--(5,-5);
\draw[thick](6,5)--(6,-5);
\draw[thick](6,5)--(7,-5);
\draw[thick](6,5)--(8,-5);
\draw[thick](6,5)--(9,-5);

\draw[thick](7,5)--(5,-5);
\draw[thick](7,5)--(6,-5);
\draw[thick](7,5)--(7,-5);
\draw[thick](7,5)--(8,-5);
\draw[thick](7,5)--(9,-5);

\draw[thick](8,5)--(5,-5);
\draw[thick](8,5)--(6,-5);
\draw[thick](8,5)--(7,-5);
\draw[thick](8,5)--(8,-5);
\draw[thick](8,5)--(9,-5);

\draw[thick](9,5)--(5,-5);
\draw[thick](9,5)--(6,-5);
\draw[thick](9,5)--(7,-5);
\draw[thick](9,5)--(8,-5);
\draw[thick](9,5)--(9,-5);

\draw[thick](0,8)--(-5,5);
\draw[thick](0,8)--(-6,5);
\draw[thick](0,8)--(-7,5);
\draw[thick](0,8)--(-8,5);
\draw[thick](0,8)--(-9,5);

\draw[thick](0,8)--(5,5);
\draw[thick](0,8)--(6,5);
\draw[thick](0,8)--(7,5);
\draw[thick](0,8)--(8,5);
\draw[thick](0,8)--(9,5);

\draw[thick](0,8)--(-2,0);
\draw[thick](0,8)--(-1,1);
\draw[thick](0,8)--(0,2);
\draw[thick](0,8)--(1,1);
\draw[thick](0,8)--(2,0);

\draw[thick](0,8)--(-1,-1);
\draw[thick](0,8)--(1,-1);

\draw[thick](0,-8)--(-1,1);
\draw[thick](0,-8)--(1,1);

\draw[thick](0,-8)--(-5,-5);
\draw[thick](0,-8)--(-6,-5);
\draw[thick](0,-8)--(-7,-5);
\draw[thick](0,-8)--(-8,-5);
\draw[thick](0,-8)--(-9,-5);

\draw[thick](0,-8)--(5,-5);
\draw[thick](0,-8)--(6,-5);
\draw[thick](0,-8)--(7,-5);
\draw[thick](0,-8)--(8,-5);
\draw[thick](0,-8)--(9,-5);

\draw[thick](0,-8)--(-2,0);
\draw[thick](0,-8)--(-1,-1);
\draw[thick](0,-8)--(0,-2);
\draw[thick](0,-8)--(1,-1);
\draw[thick](0,-8)--(2,0);

\draw[thick](-2,0)--(-5,5);
\draw[thick](-2,0)--(-6,5);
\draw[thick](-2,0)--(-7,5);
\draw[thick](-2,0)--(-8,5);
\draw[thick](-2,0)--(-9,5);

\draw[thick](2,0)--(5,5);
\draw[thick](2,0)--(6,5);
\draw[thick](2,0)--(7,5);
\draw[thick](2,0)--(8,5);
\draw[thick](2,0)--(9,5);

\draw[thick](-2,0)--(-5,-5);
\draw[thick](-2,0)--(-6,-5);
\draw[thick](-2,0)--(-7,-5);
\draw[thick](-2,0)--(-8,-5);
\draw[thick](-2,0)--(-9,-5);

\draw[thick](2,0)--(5,-5);
\draw[thick](2,0)--(6,-5);
\draw[thick](2,0)--(7,-5);
\draw[thick](2,0)--(8,-5);
\draw[thick](2,0)--(9,-5);

\end{tikzpicture}}
\end{center}

\section{Introduction} \label{sec:intro}

Supersingular elliptic curve isogeny graphs have undergone a surge of research activity in recent years, in part due to their suitability as a mathematical foundation for quantum-safe cryptographic applications. In particular, the path finding problem in these graphs seems to be intractable even on a quantum computer. For a prime $p$, the supersingular isogeny graph $\gfb$ has as its  vertex set the $\Fpbar$-isomorphism classes of supersinguar elliptic curves, labeled by their $j$-invariants in~$\mathbb{F}_{p^2}$. The directed edges of $\gfb$ are the $\ell$-isogenies between elliptic curves representing vertices, where $\ell$ is a (usually small) prime. Given two vertices, represented by two elliptic curves $E, E'$ over $\Fpbar$, the path finding problem in $\gfb$ asks to find a path from $E$ to $E'$ comprised of $\ell$-isogenies. The presumed intractability of this problem provides security for a number of cryptographic protocols, including \cite{CGL09,SqiSign,Scallop} and their variants. 

It is well-known that supersingular $\ell$-isogeny graphs are optimal expander graphs, and are in fact Ramanujan graphs when $p \equiv 1 \pmod{12}$. The security analysis of supersingular isogeny based cryptographic schemes typically assumes that $\gfb$ behaves like a ``random'' Ramanujan graph, a supposition that has since been called into question. For example, the $p$-power Frobenius acting on $\gfb$ pairs up paths with their  Frobenius conjugate paths. It also fixes vertices in $\Fp$ and those $\ell$-isogenies between them that are defined over $\Fp$. 
Moreover, path finding becomes substantially easier when the start and end vertices belong to $\Fp$ \cite{DelfGalbraith,Childs14}. Such special structural features exhibited by the subgraph of $\gfb$ induced by the $\Fp$-vertices, referred to as the \emph{spine} of $\gfb$, may make it possible to distinguish a supersingular elliptic curve $\ell$-isogeny graph from a random optimal expander or Ramanujan graph.

These questions prompted the authors of \cite{Arpin} to launch a thorough investigation into the spine $\slp$ of $\gfb$. To that end, they considered the supersingular isogeny $\Fp$-graph $\gf$, where the vertices are now $\Fp$-isomorphism classes of supersingular  elliptic curves and the edges are $\Fp$-rational $\ell$-isogenies. The structure of this graph is well understood and was described in detail in \cite{DelfGalbraith}. There is a natural two-step process of passing from $\gf$ to the spine $\slp \subset \gfb$: vertices in $\gf$ corresponding to twists of curves are identified in $\slp$ (with any $\Fp$-isogenies between them turning into loop edges), and edges arising from $\ell$-isogenies not defined over $\Fp$ are then added. In \cite{Arpin}, the possible behaviors exhibited by the connected components of $\gf$ under this process was analyzed in detail. Here, we expand on this exploration, refining the results of \cite{Arpin} and offering new findings. 

Our contribution herein is two-fold. In Section \ref{sec:ell2structurethms}, we describe all the ways, characterized by explicit congruence conditions on $p$, in which components of $\gf$ can behave when passing to $\slp$ for the arguably most interesting case of $\ell = 2$. In Section \ref{sec:ell3structurethms}, we do the same for $\ell = 3$ and give a road map of how to extend our approach to isogeny degrees $\ell \ge 5$.

Leveraging our spine structure results, we describe the diameter (the largest possible directed distance between any pair of vertices) in any component of $\slp$ in Section \ref{sec:exp}. This is followed by an extensive numerical investigation of the center of $\gfb$, i.e.\ the set of vertices for which the largest distance to any other vertex is minimal. Center vertices can be thought of as having increased connectivity to the rest of the graph compared to the vertices outside the center. The fact that Frobenius is a graph automorphism on $\gfb$ that fixes precisely the $\Fp$-vertices might suggest that spine vertices are more prominently represented in the center of $\gfb$ than vertices outside $\Fp$. Our numerical experiments for  $\ell = 2, 3$ and the first 2260 primes $p \ne 2, 3$ (i.e.\ $5 \leq p < 20000$) demonstrate that this is in fact not the case for this range of parameters, thereby providing evidence against this claim. Plots of $\Fp$-vertices counts belonging to the center of $\mathcal{G}_\ell(\Fpbar)$ exhibit a wave pattern, where wave peaks become higher and are spaced increasingly far apart for larger values of $p$. Similar wave shapes appear for $\ell = 3$, and this behavior is more pronounced when plotting the size of the entire center of $\mathcal{G}_2(\Fp)$ rather than just the number of its spine vertices. Although visually striking, this pattern is in fact evidence that the center size of $\gfb$ behaves like that of a random $(\ell +1)$-regular graph.

\subsection{Accompanying data and code}
The data used to create the figures in this work, as well as the SageMath \cite{sage} code used to generate that data, can be found at the GitHub repository  \cite{LUCANTGitHub}. Consult the \href{https://github.com/TahaHedayat/LUCANT-2025-Supersingular-Ell-Isogeny-Spine/blob/main/README.md}{README.md} file for a list of the included files and their functionalies. This code is modified from its original version, which was created by the first author in the summer of 2023. Data collected in \verb|.csv| files was generated using SageMath 10.4 \cite{sage} on a MacBook Pro, Apple M3, with 16 GB of memory, running macOS Sonoma 14.6.1. 

\subsection{Acknowledgments} The first and third author acknowledge the support of the \emph{Natural Sciences and Engineering Research Council of Canada} (NSERC). The second author is a faculty fellow of the Commonwealth Cyber Initiative and is supported by an AMS-Simons Travel Grant. We thank Jonathan Love, Jonathan Komada Eriksen, and Thomas Decru for their observations and comments in Section 6.

\section{Supersingular elliptic curve isogeny graphs} \label{sec:sig}

Let $p$ be a prime, $\Fp$ the finite field of $p$ elements and $\Fpbar$ a fixed algebraic closure of $\Fp$. We consider supersingular elliptic curves over $\Fpbar$ and recall that every $\Fpbar$-isomorphism class of such curves contains a representative that is defined over~$\mathbb{F}_{p^2}$. Isomorphism classes of supersingular elliptic curves are classified by their $j$-invariant, which is thus an element of $\mathbb{F}_{p^2}$. The $j$-invariants of curves with extra automorphisms may or may not be supersingular; specifically, $j = 1728$ is supersingular if and only if $p \equiv 3 \pmod{4}$, and $j = 0$ is supersingular if and only if $p \equiv 2 \pmod{3}$. 

While the endomorphism ring of a supersingular elliptic curve defined over $\Fpbar$ is isomorphic to a maximal order in a quaternion algebra, the ring of endomorphisms defined over $\Fp$ of any elliptic curve over $\Fp$ is isomorphic to an imaginary quadratic order. Specifically, when $p \equiv 1 \pmod{4}$, all supersingular elliptic curves defined over $\Fp$ have $\Fp$-endomorphism ring isomorphic to the maximal order $\ZZ[\sqrt{-p}]$ of discriminant $-4p$. When $p \equiv 3 \pmod{4}$, all such curves have $\Fp$-endomorphism ring isomorphic to either the maximal order $\ZZ[(1+\sqrt{-p})/2]$ of discriminant $-p$ or its index 2 suborder $\ZZ[\sqrt{-p}]$  of discriminant $-4p$. For any quadratic discriminant $\Delta$, denote by $h(\Delta)$ the class number (i.e.\ the size of the class group) of the quadratic order of discriminant $\Delta$.

\begin{proposition}[Class number, $p \equiv 3 \pmod{4}$] \textcolor{white}{.} \label{prop:classno} 
If $p \equiv 3 \pmod{4}$, then the quadratic order $\ZZ[(1+\sqrt{-p})/2]$ is maximal and has odd class number $h(-p)$, and its quadratic suborder $\ZZ[\sqrt{-p}]$ has index 2 and class number $h(-4p) = 3h(-p)$.
\end{proposition}
\begin{proof}
    Suppose $p \equiv 3 \pmod{4}$. Then $\ZZ[(1+\sqrt{-p})/2]$ has discriminant $-p$ which is fundamental, so this is the maximal order of $\mathbb{Q}(\sqrt{-p})$. The suborder $\ZZ[\sqrt{-p}]$ has discriminant $-4p$ and hence index 2 in the maximal order. 
       
    For any quadratic discriminant $\Delta < 0$, the 2-rank of the class group of $\mathbb{Q}(\sqrt{-p})$, i.e.\ its number of 2-Sylow factors, is one less than the number of prime factors of $\Delta$ by genus theory (see \cite[p.\ 170]{JW} for example). Hence $h(-p)$ is odd when $p \equiv 3 \pmod{4}$. The identity $h(-4p) = 3h(-p)$ follows from \cite[Cor.\ 7.28]{Cox}).  
\end{proof}

When $p \equiv 1 \pmod{4}$, the 2-rank of the class group of $\mathbb{Q}(\sqrt{-p})$ is 1, but this does not determine the 2-adic valuation of $h(-4p)$.

Now fix a prime $\ell$. We associate two directed graphs to the set of supersingular elliptic curves and their $\ell$-isogenies. For both graphs, the set of vertices does not depend on $\ell$, but the set of edge does. Two isogenies are said to be \emph{$\Fp$-equivalent} (resp., \emph{$\Fpbar$-equivalent}) if they are equal up to post-composition with an $\Fp$-automorphism (resp., an $\Fpbar$-autormorphism).

\begin{definition}[Supersingular elliptic curve isogeny graphs]
  
Let $p$ and $\ell$ be primes. Define the following two graphs:
\begin{itemize}
    \item The \emph{supersingular elliptic curve $\ell$ isogeny graph over $\Fp$}, denoted $\mathcal{G}_\ell(\Fp)$, is the graph whose vertices are $\Fp$-isomorphism classes of supersingular elliptic curves and whose edges are
    $\ell$-isogenies of these curves up to $\Fp$-equivalence. 
    \item The \emph{supersingular elliptic curve $\ell$-isogeny graph over $\Fpbar$}, denoted $\gfb$, is the graph whose vertices are $\Fpbar$-isomorphism classes of supersingular elliptic curves and whose edges are $\ell$-isogenies of these curves up to $\Fpbar$-equivalence. 
\end{itemize}
\end{definition}

By identifying $\ell$-isogenies with their duals, both graphs become undirected. We will work with the undirected variants when it is convenient and edge direction does not matter. 

It is well-known that $\gfb$ is an optimal expander graph, and is in fact an $(\ell+1)$-regular Ramanujan graph when $p \equiv 1 \pmod{12}$. We recall three key results governing the structure of~$\gfb$:

\begin{itemize}
    \item  For any $j$-invariant $j$, the neighbors of $j$ in $\gfb$ are precisely the roots of $\Phi_\ell(j,Y) \pmod{p}$. The multiplicity of an edge joining $j$ to $j'$ is the multiplicity of the root $j'$ of $\Phi_\ell(j,Y) \pmod{p}$. An explicit list of modular polynomials of can be found at \cite{drewmod}. 
    \item Two (not necessary distinct) $j$-invariants $j, j'$ are joined by a multi-edge in $\gfb$ if both are roots of the polynomial Res$_\ell(X) \pmod{p}$, where
    \begin{equation} \label{eq:Res}
        \mbox{Res}_\ell(X) = \mbox{Res} \left ( \Phi_\ell(X,Y), \frac{\partial}{\partial Y} \Phi_\ell(X,Y); \ Y \right ) ,
    \end{equation}
    the resultant of the the level $\ell$ modular polynomial $\Phi_\ell(X,Y)$ and its partial derivative with respect to $Y$ when both are considered as polynomials in $Y$ with coefficients in $\ZZ[X]$.
    \item A root of the Hilbert class polynomial $H_\Delta(X)$ for the imaginary quadratic order of discriminant $\Delta$ is a supersingular $j$-invariant if and only if $p$ is inert in $\mathbb{Q}(\sqrt{\Delta})$. 
\end{itemize}

The structure of $\mathcal{G}_\ell(\Fp)$ for $p \ge 5$ was first described in \cite{DelfGalbraith}.
Following volcano terminology, vertices in $\gf$ corresponding to elliptic curves with $\Fp$-endomorphism ring $\ZZ[(1+\sqrt{-p})/2]$ are said to lie on the \emph{surface}; those with $\Fp$-endomorphism rings $\ZZ[\sqrt{-p}]$ lie on the \emph{floor}. When $p \equiv 1 \pmod{4}$, all vertices are on the floor. We recall the main  structure theorem of \cite{DelfGalbraith} here. 

\begin{theorem}{\cite[Thm.\ 2.7]{DelfGalbraith}, Structure of $\gf$}\label{thm:DG}
    Let $p \ge 5$ be a prime. 
    \begin{enumerate}
        \item If $\ell = 2$ and $p\equiv 1\pmod{4}$, then $\mathcal{G}_2(\Fp)$ consists of $h(-4p)$ vertices joined in adjacent pairs.
      
        \begin{center}
        \begin{tikzpicture}
            \node[draw,fill,circle,scale=0.5] (a) at (0,0) {};
            \node[draw,fill,circle,scale=0.5] (b) at (1,0) {};
            \node[draw,fill,circle,scale=0.5] (c) at (3,0) {};
            \node[draw,fill,circle,scale=0.5] (d) at (4,0) {};
            \node[draw,fill,circle,scale=0.5] (e) at (6,0) {};
            \node[draw,fill,circle,scale=0.5] (f) at (7,0) {};
            \draw[-] (a) to (b);
            \draw[-] (c) to (d);
            \draw[-] (e) to (f);
        \end{tikzpicture}
        \end{center}
        \item If $\ell = 2$ and $p\equiv 3\pmod{8}$, then $\mathcal{G}_2(\Fp)$ consists of $4h(-p)$ vertices organized into $h(-p)$ \emph{tripod} formations. Each tripod consists of a single vertex on the surface that is adjacent to three distinct vertices on the floor. 

        \begin{center}
            \begin{tikzpicture}
                \node[circle,fill,scale=0.5] (a) at (0,0) {};
                \node[circle,fill,scale=0.5] (b) at (2,0) {};
                \node[circle,fill,scale=0.5] (a1) at (-.5,-1) {};
                \node[circle,fill,scale=0.5] (a2) at (0,-1) {};
                \node[circle,fill,scale=0.5] (a3) at (.5,-1) {};
                
                \node[circle,fill,scale=0.5] (b1) at (1.5,-1) {};
                \node[circle,fill,scale=0.5] (b2) at (2,-1) {};
                \node[circle,fill,scale=0.5] (b3) at (2.5,-1) {};

                \draw[-] (a) to (a1);
                \draw[-] (a) to (a2);
                \draw[-] (a) to (a3);
                \draw[-] (b) to (b1);
                \draw[-] (b) to (b2);
                \draw[-] (b) to (b3);
            \end{tikzpicture}
        \end{center}

        \item If $\ell = 2$ and $p\equiv 7\pmod{8}$, then $\mathcal{G}_2(\Fp)$ consists of $2h(-p)$ vertices, organized into \emph{volcanoes}.         Each volcano contains a (possibly degenerate) cycle consisting of vertices on the surface, each of which is adjacent to a unique vertex on the floor. 
        \begin{center}
        \begin{tikzpicture}
\node[circle,fill,scale=0.5] (E1) at (1/3,3/3) {};
\node[circle,fill,scale=0.5] (E1') at (1/3,1/6) {};
\node[circle,fill,scale=0.5] (E2) at (2/3,1/3) {};
\node[circle,fill,scale=0.5] (E2') at (2/3,-3/6) {};
\node[circle,fill,scale=0.5] (E3) at (3.5/3,4/3) {};
\node[circle,fill,scale=0.5] (E3') at (3.5/3,3/6) {};
\node[circle,fill,scale=0.5] (E4) at (5/3,1/3) {};
\node[circle,fill,scale=0.5] (E4') at (5/3,-3/6) {};
\node[circle,fill,scale=0.5] (E5) at (6/3,3/3) {};
\node[circle,fill,scale=0.5] (E5') at (6/3,1/6) {};

\draw[-] (E1) to (E2);
\draw[-] (E1) to (E1');
\draw[-] (E2) to (E2');
\draw[-] (E3) to (E3');
\draw[-] (E4) to (E4');
\draw[-] (E5) to (E5');
\draw[-] (E3) to (E1);
\draw[-] (E3) to (E5);
\draw[-] (E4) to (E5);
\draw[-] (E4) to (E2);

\end{tikzpicture}
\end{center}
\item If $\ell > 2$, with $\ell \ne p$, then $\gf$ is a disoint union of (possibly degenerate) cycles. If $p\equiv 3 \pmod{4}$, then each cycle contains either surface vertices or floor vertices, but not both. 

\begin{center}
\begin{tikzpicture}
\node[circle,fill,scale=0.5] (E1) at (1/3,3/3) {};
\node[circle,fill,scale=0.5] (E2) at (2/3,1/3) {};
\node[circle,fill,scale=0.5] (E3) at (3.5/3,4/3) {};
\node[circle,fill,scale=0.5] (E4) at (5/3,1/3) {};
\node[circle,fill,scale=0.5] (E5) at (6/3,3/3) {};

\draw[-] (E1) to (E2);
\draw[-] (E3) to (E1);
\draw[-] (E3) to (E5);
\draw[-] (E4) to (E5);
\draw[-] (E4) to (E2);

\end{tikzpicture}
\end{center}
\end{enumerate}

In cases (3) and (4), the length of each cycle is the order of the class generated by a prime ideal above~$\ell$ in the class group of the corresponding quadratic order. When $p \equiv 3 \pmod{4}$, vertices on the surface may be incident with loop edges. 
\end{theorem}

Note that cycles may be degenerate, i.e.\ consist of one vertex only. In particular, for $\ell > 2$, $\gf$ may consist of isolated vertices (possibly with loops); for example, when the Legendre symbol $(\frac{-p}{\ell}) = -1$ or $\mathbb{Q}(\sqrt{-p})$ has class number~1. 

We characterize loops and multi-edges in $\gf$ when $p > \ell$. 

\begin{proposition}[Loops in $\gf$] \label{prop:loops}
    Suppose $p > \ell$. 
    \begin{enumerate}
        \item If $p \equiv 3 \pmod{4}$, then a vertex on the surface of $\gf$ is incident with a loop if and only if $4\ell - p$ is a perfect square. In this case, every vertex on the surface is incident with two distinct pairs of loops, each pair corresponding to a degree $\ell$ endomorphism and its dual. 

        \begin{center}
            \begin{tikzpicture}
            \node (A) at (0, 0) [draw,fill,circle,scale=0.5] {};
            \draw (0.45,0) ellipse (0.45 and 0.225);
            \draw (0.6,0) ellipse (0.6 and 0.4);
            \draw (-0.45,0) ellipse (0.45 and 0.225);
            \draw (-0.6,0) ellipse (0.6 and 0.4);
            \end{tikzpicture}
        \end{center}
              
        \item No vertex of $\gf$ on the floor is incident with a loop.
    \end{enumerate}
\end{proposition}
\begin{proof}
    Loops in $\gf$ correspond to degree $\ell$ endomorphisms over $\Fp$, which in turn correspond to elements of norm $\ell$ in the appropriate quadratic order. 

    \begin{enumerate}
        \item For brevity, put $\omega = (1+\sqrt{-p})/2$ and let $\alpha = a + b\omega \in \ZZ[\omega]$ be the element of norm $\ell$ corresponding to loop incident with a vertex on the surface of~$\gf$. Since $p > \ell$, we obtain 
        \[ 4p > 4\ell = 4N(\alpha) = (2a+b)^2 +  b^2p \ge b^2p . \]
        so $|b| \le 1$. If $b = 0$, then $\ell = N(\alpha) = a^2$ which is impossible since $\ell$ is prime. This forces $b = \pm 1$, so $4\ell - p = (2a\pm 1)^2$ is a perfect square. Conversely, if  $4\ell - p = m^2$ with $m \in \ZZ$, then $m$ must be odd. Thus, $\alpha = a + \omega \in \ZZ[\omega]$ with $a = (m-1)/2$ is an element of norm~$\ell$. Note that $\alpha = (m+\sqrt{-p})/2$, so the 4 loops at $j$ correspond to the 4 elements $(\pm m \pm \sqrt{-p})/2 \in \ZZ[\omega]$ of norm $\ell$. 

        \item Now let $\alpha = a + b\sqrt{-p} \in \ZZ[\sqrt{-p}]$. Then $N(\alpha) = a^2 + b^2 p$. If $b \ne 0$, then $N(\alpha) \ge p > \ell$, wheres if $b = 0$, then $N(\alpha) = a^2$ which is not prime and hence also distinct from $\ell$. Hence $\ZZ[\sqrt{-p}]$ contains no elements of norm~$\ell$. \qedhere
    \end{enumerate}
\end{proof}

In particular, $\gf$ can only contain loops if $p \equiv 3 \pmod{4}$ and $p \le 4\ell - 1$. Applying this to $\ell = 2, 3$ immediately yields the following.

\begin{corollary}[Loops in $\gf$ for $\ell = 2, 3$] \label{cor:loops23}
For $\ell = 2, 3$ and $p > \ell$, the graph $\gf$ contains no loops except when $(\ell, p) \in \{ (2,7), (3, 11)\}$, where $\gf$ has one vertex that is incident with two pairs of loops. 
\end{corollary}

We now turn to multi-edges in $\gf$ that are not loops. 

\begin{proposition}[Multi-edges in $\gf$] \label{prop:multi-edges}
    Suppose $p > \ell$. 
    \begin{enumerate}
        \item  If $p \equiv 3 \pmod{4}$, then $\gf$ contains no directed multi-edges except possibly loops unless $(\ell, p) = (2,3)$. Here, $\mathcal{G}_2(\mathbb{F}_3)$ has one directed triple edge. 
        
         \begin{center}
            \usetikzlibrary {arrows.meta}
            \begin{tikzpicture}
                \node (A) at (-1.5, 0) [draw,fill,circle,scale=0.5] {};
                \node (B) at (1.5, 0) [draw,fill,circle,scale=0.5] {};
                \draw[->, >=latex, bend left=28] (A) to (B);  
                \draw[->, >=latex, bend left=8] (A) to (B); 
                \draw[->, >=latex, bend right=-28] (B) to (A); 
                \draw[->, >=latex, bend right=8] (A) to (B); 
            \end{tikzpicture}
        \end{center}
        \item If $p \equiv 1 \pmod{4}$, then $\gf$ contains directed multi-edges if and only if it has at least two vertices and $2\ell - p$ is a perfect square. In this case, the vertices in $\gf$ are joined pairwise by pairs of double edges in opposite directions. 
       
        \begin{center}
            \usetikzlibrary {arrows.meta}
            \begin{tikzpicture}
                \node (A) at (-1.5, 0) [draw,fill,circle,scale=0.5] {};
                \node (B) at (1.5, 0) [draw,fill,circle,scale=0.5] {};
                \draw[->, >=latex, bend left=28] (A) to (B);  
                \draw[->, >=latex, bend left=8] (A) to (B); 
                \draw[->, >=latex, bend right=-28] (B) to (A); 
                \draw[->, >=latex, bend right=-8] (B) to (A); 
            \end{tikzpicture}
        \end{center}
    \end{enumerate}
   
\end{proposition}
\begin{proof}
   A directed multi-edge in $\gf$ that is not a loop corresponds to multiple $\Fp$-non-equivalent $\ell$-isogenies $\phi, \psi : E \rightarrow E'$ where $E$ and $E'$ are not isomorphic over $\Fp$. Then $\hat{\phi}\psi$ is an endomorphism on $E$ of degree $\ell^2$, and we have a multi-edge if and only if $\hat{\phi}\psi$ is not the multiplication by $\ell$ map on $E$. Thus, multi-edges arise from elements $\alpha \ne \pm \ell$ of norm $\ell^2$ in the corresponding quadratic order which are not squares up to sign.     
    
    \begin{enumerate}
        \item Assume $p \equiv 3 \pmod{4}$. Note that $\ell$ does not ramify in $\ZZ[\omega]$ as $\ell \ne p$. Suppose first that $\ell$ splits in $\ZZ[\omega]$, and write $(\ell) = \mathfrak{l}\overline{\mathfrak{l}}$ with prime ideals~$\mathfrak{l}, \overline{\mathfrak{l}}$ of $\ZZ[\omega]$. Note that this precludes the case $p = 3$, as $\ell < p$ implies $\ell = 2$ in this case, but 2 is inert in $\mathbb{Q}(\sqrt{-3})$. Let $\alpha \in \ZZ[\omega]$ with $\alpha \ne \pm \ell$ and $N(\alpha) = \ell^2$. Unique prime ideal factorization, together with $(\alpha) \ne (\ell)$, forces $(\alpha) = \mathfrak{l}^2$ or $(\alpha)= \overline{\mathfrak{l}}^2$. Assume the former; the case $(\alpha) = \overline{\mathfrak{l}}^2$ is entirely analogous. Then~$\mathfrak{l}^2$ is principal. Since  $h(-p)$ is odd by Proposition \ref{prop:classno}, $\mathfrak{l}$~must be principal. This means that $\alpha$ is a square in $\ZZ[\omega]$ up to sign, which we precluded. 

        Now assume that $\ell$ is inert in $\ZZ[\omega]$. Since $(\ell)$ is the only prime ideal of norm $\ell^2$, unique prime ideal factorization implies $(\alpha) = (\ell)$. The assumption $\alpha \ne \ell$ now forces $p = 3$ (the only case where $\ZZ[\omega]$ has non-trivial units) and hence $\ell = 2$ since $\ell < p$. Thus, $\alpha = \pm 2 \zeta^k$ where $k = 1, 2$ and $\zeta$ is a primitive cube root of unity, so $\pm \alpha \in \{ 1 + \sqrt{-3}, 1 - \sqrt{-3} \}$. Indeed, among the four $\mathbb{F}_3$-isomorphism classes of supersingular elliptic curves over $\mathbb{F}_3$, exactly two, represented by the curves  $E_{\pm} : y^2 = x^3 \pm x$, are 2-isogenous over~$\mathbb{F}_3$. The curve $E_+$ is on the floor, whereas $E_-$ is on the surface and has, up to sign, three $\mathbb{F}_3$-rational automorphisms $(x,y) \mapsto (x+i, y)$ for $i = 0, 1, 2$. So there are three $\mathbb{F}_3$-inequivalent 2-isogenies from $E_-$ to $E_+$, producing three directed edges. Their duals differ only by post-composition by an automorphism  on $E_-$ and are hence equivalent, yielding one edge in the opposite direction. 

        \item  Now suppose $p \equiv 1 \pmod{4}$. Let $\alpha \in \ZZ[\sqrt{-p}]$ with $N(\alpha) = \ell^2$, $\alpha \ne \pm \ell$ and $\pm \alpha$ not a square in $\ZZ[\sqrt{-p}]$. Write $\alpha = a+b\sqrt{-p}$ with $a, b \in \ZZ$. Then $a^2 + b^2 p = \ell^2$ and $b \ne 0$, so $|a| < \ell$. We have
        \begin{equation} \label{eq:p}
            b^2 p = \ell^2 - a^2 = (\ell-|a|)(\ell+|a|) ,  
        \end{equation}
        so $p$ divides $\ell-|a|$ or $\ell+|a|$. Since $1 \le \ell - |a| \le \ell < p$, the first of these possibilities cannot happen; hence $p \mid \ell + |a|$. Write $\ell + |a| = kp$ for some $k \in \ZZ$. Then $k \ge 1$ and $k\ell < kp = \ell + |a| < 2\ell$, forcing $k = 1$ and hence $p = \ell + |a|$. By \eqref{eq:p}, we have $b^2 = \ell - |a| = 2\ell - p$, so $2\ell - p$ is a perfect square. Conversely, if $2\ell - p = b^2$ for some $b \in \ZZ$, then the 4 elements $\alpha = \pm (p-\ell) \pm b\sqrt{-p}$ all have norm $\ell^2$, are distinct from $\pm\ell$, and are not squares in $\ZZ[\sqrt{-p}]$.    
    
        Note that as in the case $p \equiv 3 \pmod{4}$, we have $(\alpha) = \mathfrak{l}^2$ or $(\alpha) = \overline{\mathfrak{l}}^2$, where  $\mathfrak{l}$ and $\overline{\mathfrak{l}}$ are the two prime ideals above $(\ell)$ in $\ZZ[\sqrt{-p}]$. So the ideal class of $\mathfrak{l}$ has order 2 in the class group of $\mathbb{Q}(\sqrt{-p})$. This means that if $\gf$ has more than one vertex and $2\ell - p$ is a perfect square (which rules out $\ell = 2$ as $p \ge 5$), all the cycles in $\gf$ as described in part (4) of Theorem \ref{thm:DG} have length 2. They correspond to two pairs of directed edges and their respective duals. \qedhere
     \end{enumerate}
\end{proof}   

Proposition \ref{prop:multi-edges} shows that if $p > \ell > 2$, then $\gf$ can only contain directed double edges when $p \equiv 1 \pmod{4}$ and $p \le 2\ell-1$. The cases $\ell = 2, 3$ can now once again be easily deduced.

\begin{corollary}[Multi-edges in $\gf$ for $\ell = 2, 3$] \label{cor:multi23}
For $\ell = 2, 3$  and $p > \ell$, apart from the loops of Corollary \ref{cor:loops23}, $\gf$ contains no directed multi-edges except when $(\ell, p) \in \{ (2,3), (3,5) \}$. 
\end{corollary}

A list of prime pairs $(\ell, p)$ with $2 \le \ell < 100$ and $p > \ell$ for which $\gf$ has directed multi-edges, including loops, can be found in the file name \href{https://github.com/TahaHedayat/LUCANT-2025-Supersingular-Ell-Isogeny-Spine/blob/main/loops%20%26%20multi-edges%20in%20G_l(Fp).pdf}{\texttt{loops \& multi-edges in G\_l(Fp).pdf}} at the GitHub repository \cite{LUCANTGitHub}. 
A notebook with code to generate visual depictions of $\gf$ and $\gfb$, entitled \href{https://github.com/TahaHedayat/LUCANT-2025-Supersingular-Ell-Isogeny-Spine/blob/main/Graph_Viz.ipynb}{\texttt{Graph\_Viz.ipynb}}, is also available at \cite{LUCANTGitHub}.

\section{The spine of $\gfb$} \label{sec:spine}

In this section, we review the relationship between $\gf$ and $\gfb$; specifically the process of moving from $\Fp$-isomorphism classes of elliptic curves to $\Fpbar$-isomorphism classes, and from $\ell$-isogenies defined over $\Fp$ to those defined over $\Fpbar$. This material is a summary of \cite{Arpin}. 

\begin{definition}[Spine]
    The \emph{spine} $\slp$ is the subgraph of $\gfb$ induced by the vertices in $\Fp$. Specifically, the vertices of $\slp$ are the $\Fpbar$-isomorphism classes of supersingular elliptic curves defined over $\Fp$, and its edges are all the $\ell$-isogenies joining these vertices.
\end{definition}

There are natural maps between the graphs $\gf$, $\gfb$, and $\slp$: 

\begin{definition}\label{def:gftospfunctions} \label{def:maps}
    Define $\Gamma: \gf \rightarrow \gfb$ to take vertices of $\gf$ to their $\Fpbar$-isomorphism classes  and edges to their $\Fpbar$-equivalence classes.
    
    Next, define $\Theta: \mbox{Im}(\Gamma) \rightarrow \gfb$ to add edges between vertices in Im$(\Gamma)$ that correspond to isogenies defined over $\Fpbar$ and not defined over $\Fp$. In particular, $\Theta(\text{Im}(\Gamma)) = \slp$.
    
    Lastly, define $\Omega = \Theta \circ \Gamma: \gf \rightarrow \slp\subseteq \gfb$.
\end{definition}

When $\Omega$ is applied to $\gf$, the following graph structural changes are possible:
\begin{definition} \textcolor{white}{.} \label{def:s-f-ea-va}
    \begin{itemize}
        \item \textbf{Stacking:} Two connected components of $\gf$ \emph{stack} under the map $\Gamma$ of Definition \ref{def:maps} if they are the same graph when labeled by $j$-invariants.

        \begin{center}
            \begin{tikzpicture}
                \node[circle,draw,scale=0.5] (a) at (0,0) {$j_1$};
                \node[circle,draw,scale=0.5] (a1) at (-.5,-1) {$j_2$};
                \node[circle,draw,scale=0.5] (a2) at (0,-1) {$j_3$};
                \node[circle,draw,scale=0.5] (a3) at (.5,-1) {$j_4$};
                \node[circle,draw,scale=0.5] (b) at (2,0) {$j_1$};
                \node[circle,draw,scale=0.5] (b1) at (1.5,-1) {$j_2$};
                \node[circle,draw,scale=0.5] (b2) at (2,-1) {$j_3$};
                \node[circle,draw,scale=0.5] (b3) at (2.5,-1) {$j_4$};

                \draw[-] (a) to (a1);
                \draw[-] (a) to (a2);
                \draw[-] (a) to (a3);
                \draw[-] (b) to (b1);
                \draw[-] (b) to (b2);
                \draw[-] (b) to (b3);
                \draw[->] (3,-.5) to (4,-.5);
                \node[] (label) at (3.5,0) {$\Gamma$};

                \node[circle,draw,scale=0.5] (c) at (5,0) {$j_1$};
                \node[circle,draw,scale=0.5] (c1) at (4.5,-1) {$j_2$};
                \node[circle,draw,scale=0.5] (c2) at (5.5,-1) {$j_3$};
                \node[circle,draw,scale=0.5] (c3) at (5,-1) {$j_4$};
                
                \draw[-] (c) to (c1);
                \draw[-] (c) to (c2);
                \draw[-] (c) to (c3);
            \end{tikzpicture}
        \end{center}
        
        \item \textbf{Folding:} A connected component of $\gf$ \emph{folds} under the map $\Gamma$ if it only contains vertices corresponding to both quadratic twists for every $j$-invariant appearing as a vertex in the component.
        
        \begin{center}
        \begin{tikzpicture}
            \node[draw,circle,scale=0.5] (a) at (0,0) {$j_1$};
            \node[draw,circle,scale=0.5] (b) at (1,0) {$j_1$};

            \draw[->] (2,0) to (3,0);
            \node[] (label) at (2.5,0.5) {$\Gamma$};

            \draw[-] (a) to (b);
            \node[draw,circle,scale=0.5] (anew) at (4,0) {$j_1$};
            \draw[-,loop above] (anew) to (anew);
        \end{tikzpicture}
        \end{center}

        \item \textbf{Attachment at a vertex:} Two components of $\gf$ have a \emph{vertex attachment} under the map $\Gamma$ if they both contain a vertex with the same $j$-invariant but the neighbors of that shared $j$-invariant are not the same in the two components.

        \begin{center}
            \begin{tikzpicture}
            \node[circle,draw,scale=0.5] (E1) at (1/3,3/3) {$j_3$};
            \node[circle,draw,scale=0.5] (E2) at (2/3,1/3) {$j_3$};
            \node[circle,draw,scale=0.5] (E3) at (3.5/3,4/3) {$j_2$};
            \node[circle,draw,scale=0.5] (E4) at (5/3,1/3) {$j_2$};
            \node[circle,draw,scale=0.5] (E5) at (6/3,3/3) {$j_1$};
            
            \node[circle,draw,scale=0.5] (F1) at (1/3+2.5,3/3) {$j_1$};
            \node[circle,draw,scale=0.5] (F2) at (2/3+2.5,1/3) {$j_4$};
            \node[circle,draw,scale=0.5] (F3) at (3.5/3+2.5,4/3) {$j_4$};
            \node[circle,draw,scale=0.5] (F4) at (5/3+2.5,1/3) {$j_5$};
            \node[circle,draw,scale=0.5] (F5) at (6/3+2.5,3/3) {$j_5$};
            
            \draw[-] (E1) to (E2);
            \draw[-] (E3) to (E1);
            \draw[-] (E3) to (E5);
            \draw[-] (E4) to (E5);
            \draw[-] (E4) to (E2);
            
            \draw[-] (F1) to (F2);
            \draw[-] (F3) to (F1);
            \draw[-] (F3) to (F5);
            \draw[-] (F4) to (F5);
            \draw[-] (F4) to (F2);
            
            \draw[->] (5,.75) to (6,.75);
            \node[] (label) at (5.5,1.25) {$\Gamma$};
            
            \node[circle,draw,scale=0.5] (G3) at (1/3+6,3/3) {$j_3$};
            \node[circle,draw,scale=0.5] (G2) at (3.5/3+6,4/3) {$j_2$};
            \node[circle,draw,scale=0.5] (G1) at (6/3+ 6,3/3) {$j_1$};
            \node[circle,draw,scale=0.5] (G5) at (9+2/3,3/3) {$j_5$};
            \node[circle,draw,scale=0.5] (G4) at (6/3+ 6+5/6,4/3) {$j_4$};
            
            \draw[-] (G3) to (G2);
            \draw[-] (G1) to (G2);
            \draw[-] (G1) to (G4);
            \draw[-] (G4) to (G5);
            
            \draw[-,loop above] (G3) to (G3);
            \draw[-,loop above] (G5) to (G5);

            \end{tikzpicture}
            \end{center}

        \item \textbf{Attachment by a new edge:} Two connected components of $\gf$ have an \emph{edge attachment} if a new edge appears under the map $\Theta$ of Definition~\ref{def:maps} which connects these two components.

         \begin{center}
            \begin{tikzpicture}
                \draw[->] (-1.5,-.5) to (-.5,-.5);
                \node[] (label) at (-1,0) {$\Gamma$};
                
                \node[circle,draw,scale=0.5] (a) at (0,0) {$j_1$};
                \node[circle,draw,scale=0.5] (a1) at (-.5,-1) {$j_2$};
                \node[circle,draw,scale=0.5] (a2) at (0,-1) {$j_3$};
                \node[circle,draw,scale=0.5] (a3) at (.5,-1) {$j_4$};
                \node[circle,draw,scale=0.5] (b) at (2,0) {$j_5$};
                \node[circle,draw,scale=0.5] (b1) at (1.5,-1) {$j_6$};
                \node[circle,draw,scale=0.5] (b2) at (2,-1) {$j_7$};
                \node[circle,draw,scale=0.5] (b3) at (2.5,-1) {$j_8$};

                \draw[-] (a) to (a1);
                \draw[-] (a) to (a2);
                \draw[-] (a) to (a3);
                \draw[-] (b) to (b1);
                \draw[-] (b) to (b2);
                \draw[-] (b) to (b3);
                \draw[->] (3,-.5) to (4,-.5);
                \node[] (label) at (3.5,0) {$\Theta$};

                \node[circle,draw,scale=0.5] (c) at (5,0) {$j_1$};
                \node[circle,draw,scale=0.5] (c1) at (4.5,-1) {$j_2$};
                \node[circle,draw,scale=0.5] (c2) at (5.5,-1) {$j_3$};
                \node[circle,draw,scale=0.5] (c3) at (5,-1) {$j_4$};

                \node[circle,draw,scale=0.5] (d) at (7,0) {$j_5$};
                \node[circle,draw,scale=0.5] (d1) at (6.5,-1) {$j_6$};
                \node[circle,draw,scale=0.5] (d2) at (7.5,-1) {$j_7$};
                \node[circle,draw,scale=0.5] (d3) at (7,-1) {$j_8$};
                
                \draw[-] (c) to (c1);
                \draw[-] (c) to (c2);
                \draw[-] (c) to (c3);
                \draw[-] (d) to (d1);
                \draw[-] (d) to (d2);
                \draw[-] (d) to (d3);
                \draw[-] (c2) to [out=30,in=150] (d1);
                \draw[-] (c2) to [out=330,in=210] (d1);
            \end{tikzpicture}
        \end{center}

\end{itemize}
\end{definition}

Definition \ref{def:s-f-ea-va} immediately implies the following. 

\begin{lemma}\label{lem:foldingorstacking}
    Folding and stacking are mutually exclusive. 
\end{lemma}
\begin{proof}
    If two components stack, then one component's vertices correspond to the quadratic twists of the other component's vertices. If a component folds, then the quadratic twists of its vertices belong to that same component.
\end{proof}

As shown in \cite[Prop.\ 3.16 and Cor.\ 3.24]{Arpin}, vertex attachment is only possible at the $j$-invariant $1728$ and only for $\ell > 2$. Attachment by a new edge implies a double edge in $\gfb$ by \cite[Cor.\ 3.15]{Arpin}. We recall the precise theorems from \cite{Arpin} describing the changes to $\gf$ under the map $\Omega$:

\begin{theorem}[{\cite[Thm.\ 3.29]{Arpin}}]
    Let $\ell = 2$. Under the map $\Gamma:\mathcal{G}_2(\Fp)\to \mathcal{G}_2(\Fpbar)$ of Definition~\ref{def:gftospfunctions}, only stacking and folding are possible. Under the map $\Theta:\operatorname{Im}(\Gamma)\to\mathcal{G}_2(\Fpbar)$, at most one attachment by a new edge is possible. In particular, attachment by a vertex is not possible.
\end{theorem}

\begin{theorem}[{\cite[Thm.\ 3.18]{Arpin}}] \label{thm:3.18A+}
   Let $p$ and $\ell>2$ be distinct primes such that the order of a prime ideal $\mathfrak{l}$ above $\ell$ in the class group of $\mathbb{Q}(\sqrt{-p})$ is odd. Under the map $\Gamma:\mathcal{G}_\ell(\Fp)\to\mathcal{G}_\ell(\Fpbar)$:
   \begin{itemize}
       \item the two components containing vertices corresponding to $j = 1728$ fold and attach at the vertex $j = 1728$;
       \item all other components stack.
   \end{itemize}
   Under the map $\Theta:\operatorname{Im}(\Gamma)\to \mathcal{G}_\ell(\Fpbar)$, the number of new edges is bounded by the degree of    Res$_\ell(X) \pmod{p}$, with Res$_\ell(X)$ given in \eqref{eq:Res}.
\end{theorem}

In the next two sections, building on the results of \cite[Sec.\ 3]{Arpin}, we  explicitly describe the structure of $\mathcal{S}_2^p$ and and $\mathcal{S}_3^p$ in terms of specific congruence conditions on~$p$. We also provide a road map for extending this approach in principle to any $\ell$ and outline obstacles one might encounter. Small primes are treated separately elsewhere: a detailed description of the graphs $\gf$, $\slp$ and $\gfb$ for $2 \le p \le 13$ can be found under the file name \href{https://github.com/TahaHedayat/LUCANT-2025-Supersingular-Ell-Isogeny-Spine/blob/main/SmallCharacteristicGraphDescription.pdf}{\texttt{Small\-Characteristic\-Graph\-Description.pdf}}. This information can also be generated with the notebook \href{https://github.com/TahaHedayat/LUCANT-2025-Supersingular-Ell-Isogeny-Spine/blob/main/Small_Prime_Information.ipynb}{\texttt{Small\_Prime\_Information.ipynb}}, and the notebook \href{https://github.com/TahaHedayat/LUCANT-2025-Supersingular-Ell-Isogeny-Spine/blob/main/Graph_Viz.ipynb}{\texttt{Graph\_Viz.ipynp}} generates images of all three graphs. All these sources are available at \cite{LUCANTGitHub}.

\section{Structure of the spine $\mathcal{S}_2^p$} \label{sec:ell2structurethms} 

In this section, we provide congruence conditions that govern the structure of~$\mathcal{S}_2^p$, the spine for $\ell = 2$. This case carries the most interest, especially when $p \equiv 3 \pmod{4}$, due to the volcano structure of $\mathcal{G}_2(\Fp)$. For much of this section, we only consider primes $p \ge 17$; for details on the primes $2 \le p \le 13$, consult the sources cited at the end of Secton~\ref{sec:spine}.

As expected, our investigation makes extensive use of the modular poynomial $\Phi_2(X,X)$ and the polynomial Res$_2(X)$ as defined in \eqref{eq:Res}. These polynomials are given as follows:
\begin{align}
    \Phi_2(X,X) &= -(X - 1728)(X - 8000)(X + 3375)^2, \label{eq:Phi2} \\
    \mbox{Res}_2(X) &= -2^2 X^2 (X - 1728) (X + 3375)^2 (X^2 + 191025X - 121287375)^2. \label{eq:Res01}
\end{align}

By Corollary \ref{cor:loops23}, $\mathcal{G}_2(\Fp)$ contains loops only for $p = 7$.  We recall a well-known result about loops in $\mathcal{G}_2(\Fpbar)$.

\begin{lemma}[Loops in $\mathcal{G}_2(\Fpbar)$]\label{lem:loops2}
Let $p \ne 2, 7$. Loops occur in $\mathcal{G}_2(\Fpbar)$ at vertices corresponding to precisely the following $j$-invariants, all belonging to $\Fp$:
\begin{align*}
    1728 &\text{ if }p\equiv 3\pmod{4}\\
    8000 & \text{ if }p\equiv 5,7\pmod{8}\\
    -3375 &\text{ if }p\equiv 3,5,6\pmod{7}.
\end{align*}
\end{lemma}
\begin{proof}
    The loops in  $\mathcal{G}_2(\Fpbar)$] are precisely the roots of the polynomial $\Phi_2(X,X)$ of \eqref{eq:Phi2}. Its linear factors are the Hilbert class polynomials $H_{-4}(X)$, $H_{-8}(X)$ and $H_{-7}(X)$. The congruence conditions on $p$ characterize when $p$ is inert in the corresponding imaginary quadratic fields. 
\end{proof}

    The $j$-invariants listed in Lemma~\ref{lem:loops2} need not be distinct for primes $p\leq 5$; these small primes are handled explicitly in the document \href{https://github.com/TahaHedayat/LUCANT-2025-Supersingular-Ell-Isogeny-Spine/blob/main/SmallCharacteristicGraphDescription.pdf}{\texttt{Small\-Characteristic\-Graph\-Description.pdf}} at \cite{LUCANTGitHub}.
    
\begin{proposition}[Folding for $\ell = 2$]\label{prop:folding}
The only connected components of $\mathcal{G}_2(\Fp)$ that fold are those containing vertices with $j = 8000$ (if $p\equiv 5,7\pmod{8}$) and/or those containing vertices with $j = 1728$ (if $p\equiv 3\pmod{4}$).
\end{proposition}
\begin{proof}
    A component folds if and only if it contains only vertices corresponding to both $\Fp$-twists of a supersingular elliptic curve $j$-invariant; see \cite[Cor. 3.28]{Arpin}. The congruence class for $j = 8000$ has been updated to specify the supersingular primes for this $j$-invariant. In \cite{Arpin}, the authors mistakenly declare that $j = 8000$ is only supersingular for $p\equiv 5\pmod{8}$, when in fact $j = 8000$ is supersingular for $p\equiv 5,7\pmod{8}$. In fact, the two models that the authors list for the curve over $\mathbb{Z}$:
    \[E_{8000}:y^2 = x^3 - 4320x +96768, \quad E_{8000}^t:y^2 = x^3 - 17280x - 774144\]
    are twists by $\sqrt{-2}$ (not $\sqrt{2}$, as stated by the authors). Since $-2$ is not a square modulo $p$ for $p\equiv 5,7\pmod{8}$, $E$ and $E'$ reduce to supersingular quadratic twists over $\Fp$ for such primes $p$. Moreover, as noted in \cite{Arpin}, these curves have a $\ZZ$-rational $2$-isogeny between them, which means they belong to the same connected component of $\mathcal{G}_2(\Fp)$. 
    
    Likewise, $\ZZ$-models for $j = 1728$ show that both $\Fp$-isomorphism classes lie on the same connected component of $\mathcal{G}_2(\Fp)$. 
\end{proof}

Since edge attachment forces a double edge, we need to identify double edges in~$\mathcal{G}_2(\Fpbar)$ that join two vertices in $\Fp$. This amounts to ascertaining when the roots of the polynomial Res$_2(X)$ in \eqref{eq:Res01} are supersingular and belong to $\Fp$. For $j=0$, this requires $p \equiv 2 \pmod{3}$, and for the roots 1728 and $-3375$ of Res$_2(X)$, this was addressed in Lemma \ref{lem:loops2}. So we need only consider the quadratic factor of Res$_2(X)$, which is in fact the Hilbert class polynomial $H_{-15}(X)$. 

\begin{lemma}[Multi-edges in $\mathcal{G}_2(\Fpbar)$] \label{lem:doubleedge2}
    Let $p \ge 7$. In addition to the loops identified in Lemma \ref{lem:loops2}, $\mathcal{G}_2(\Fpbar)$ has multi-edges when $p = 7, 13$ or $p \equiv 11, 14 \pmod{15}$. 
\end{lemma}

\begin{proof}
    By \eqref{eq:Res01}, additional double edges correspond to the roots modulo $p$ of $H_{-15}(X) = X^2 + 191025X - 121287375$, which are  $(-191025 \pm 85995 \sqrt{5})/2$. For $p$ odd,  their reductions modulo $p$ belong to $\Fp$ if and only if $p \mid  85995 = 3^3 \cdot 5 \cdot 7^2 \cdot 13$ or 5 is a quadratic residue modulo $p$. So assuming $p \ge 7$, $H_{-15} \pmod{p}$ has roots in $\Fp$ if and only if $p = 7,13$ or $(\frac{5}{p}) = 1$. The latter condition holds if and only if $p \equiv \pm 1 \pmod{5}$.

    The roots of $H_{-15}(X) \pmod{p}$ are $j$-invariants of supersingular elliptic curves if and only if $(\frac{-15}{p}) = -1$. Note that this holds for $p = 7$ and $p =  13$, in which case $H_{-15}(X)$ has the double root $-191025/2 \in \Fp$; else it has two distinct roots in $\Fp$. Assuming $(\frac{5}{p}) = 1$, the condition $(\frac{-15}{p}) = -1$ reduces to $(\frac{-3}{p}) = -1$, or equivalently, $p \equiv 2 \pmod{3}$. Finally, $p \equiv \pm 1 \pmod{5}$ and $p \equiv 2 \pmod{3}$ if and only if $p \equiv 11 \mbox{ or } 14 \pmod{15}$. 
\end{proof}

\begin{proposition}[New edges and edge attachment for $\ell = 2$] \label{prop:EApnot7mod8}
    Suppose $p \ge 17$.
    \begin{enumerate}
        \item  If $p\not\equiv 7\pmod{8}$, then new edges appear at the supersingular $j$-invariants which are roots of $H_{-15}(X)$. The new edge joining the two distinct roots of $H_{-15}(X)$ is attaching. Thus, edge attachment happens when $p\equiv 11$, 29, 41, 59, 89, 101 $\pmod{120}$. 

        \item If $p\equiv 7\pmod{8}$, then attachment by an edge can only happen between vertices distinct from  $-3375$, 1728 and 0 whose $j$-invariants are roots of~$H_{-15}(X)$.
    \end{enumerate}
\end{proposition}
\begin{proof}
    This is an extension of \cite[Cor.\ 3.30]{Arpin}, where the authors proved the result under the assumption that $p > 101$. 
    
    New edges correspond to distinct roots of $H_{-15}(x)$ that are supersingular $j$-invariants in $\Fp$. Since $p \ne 7, 13$, the roots of $H_{-15}(x)$ are distinct modulo $p$ and are supersingular for $p \equiv 11$ or $14 \pmod{15}$ by Lemma~\ref{lem:doubleedge2}. Combining these congruence conditions with $p\not\equiv 7\pmod{8}$ results in the set of congruence classes listed in part~(1). In this case, if $p > 101$, then the new edge is an attaching edge by~\cite[Cor.\ 3.30]{Arpin}. It thus suffices to verify this for $p = 29, 41, 59, 71, 89,$ and $101$. For $p = 29, 41$ and 59, edge attachment holds by the connectivity of $\mathcal{G}_2(\Fpbar)$ and the fact that all the supersingular $j$-invariants belong to $\Fp$. For $p = 89$ and~$101$, direct computation of $\mathcal{S}_2^p$ confirms the edge attachment.

    The case $p \equiv 7 \pmod{8}$ is covered in~\cite[Prop.\ 3.25]{Arpin}.
\end{proof}

 For $p \equiv 11, 14 \pmod{15}$ with $p \equiv 7 \pmod{8}$, or equivalently, $p\equiv 71$ or $119 \pmod{120}$, the new edges of Proposition \ref{prop:EApnot7mod8} may or may not be attaching.

Triple edges in $\gfb$ also play a role in ascertaining edge attachment. The following theorem characterizes the occurrence of triple edges in $\mathcal{G}_2(\Fpbar)$.
 
\begin{proposition}[Triple edges in $\mathcal{G}_2(\Fpbar)$] \label{prop:TripleEdges}
For brevity, let $\mathcal{P}_2 = \{ 2,3,5,7,13 \}$. 
\begin{enumerate}
\item If $p\in \mathcal{P}_2$, then $\mathcal{G}_2(\Fpbar)$ is identical to the spine $\slp[2]$ and consists of a single vertex with a triple loop. 
\item If $p \not\in \mathcal{P}_2$, and $p\equiv 2\pmod{3}$, then there is a triple edge from the vertex corresponding to $j=0$ to the vertex corresponding to $j=54000$. 
\item For all other primes $p$, $\mathcal{G}_2(\Fpbar)$ does not contain triple edges. 
\end{enumerate}
\end{proposition}

\begin{proof}
   A proof for the primes $p\in \mathcal{P}_2$ can be found in the document \href{https://github.com/TahaHedayat/LUCANT-2025-Supersingular-Ell-Isogeny-Spine/blob/main/SmallCharacteristicGraphDescription.pdf}{\texttt{Small\-Characteristic\-Graph\-Description.pdf}} at \cite{LUCANTGitHub}, so suppose $p \notin \mathcal{P}_2$.  
   The primes $p$ for which $\slp[2]$ contains a triple edge are those for which the polynomials Res$_2(X)$ of \eqref{eq:Res01} and the polynomial
   \begin{align*}
       \text{Res}_2^{(2)}(X) &=  \mbox{Res} \left ( \Phi_\ell(X,Y), \frac{\partial^2}{\partial Y^2} \Phi_\ell(X,Y); \ Y \right ) \\ &= -2^2 \cdot 3 X(X - 405) (X^2 - 2571 X + 1492425)
   \end{align*}
   share a common root in $\Fp$. The respective sets of roots of these two polynomials are
    \[ \left \{0,1728,-3375,\frac{-191025 \pm 85995\sqrt{5}}{2} \right \} , \quad  \left \{0,405,\frac{2571\pm 39\sqrt{421}}{2} \right \}.\]
    The $j$-invariant $0$ is always a common root of both polynomials, so whenever $j=0$ is supersingular (i.e.\ whenever $p\equiv 2\pmod{3}$), the graph $\slp[2]$ has a triple edge. Solving $\Phi_2(0,Y) = 0$ for $Y$ over $\ZZ$, we see that this triple edge is to the vertex $j = 54000$ and is hence not a loop (as $54000 \not\equiv 0 \pmod{p}$ for $p \ge 7$). We show that Res$_2(X)$ and $\text{Res}_2^{(2)}(X)$ have no other shared roots when $p \not\in \mathcal{P}_2$, thus ruling out any other triple edges in~$\mathcal{S}_2^p$. 

    For brevity, let $f(X) = X^2 - 2571X + 1492425$ denote the quadratic factor of $\text{Res}_2^{(2)}(X)$. We first observe that the root $j= 1728$ of Res$_2(X)$ is not a root of $\text{Res}_2^{(2)}(X)$. For $p \not\in \mathcal{P}_2$, we note that $1728 \not \equiv 0, 405 \pmod{p}$ and $f(1728) = 3^6 \cdot 7^2 \not\equiv 0 \pmod{p}$. Similarly, the root $j = -3375$ of Res$_2(X)$ is not also a root of $\text{Res}_2^{(2)}(X)$, as $-3375 \not\equiv 0$ or 405 modulo $p$ and $f(-3375) = 3^6 \cdot 5^2 \cdot 7 \cdot 13^2 \not \equiv 0 \pmod{p}$ when  $p \not\in \mathcal{P}_2$.
    
    Finally, we establish that $H_{-15}(X)$ and $f(X)$ have no shared root. Equating the roots of these two polynomials modulo $p$ yields the following sequence of implications:
    \begin{align*}
        \frac{-191025 \pm 85995\sqrt{5}}{2} &\equiv \frac{2571 \pm 39 \sqrt{421}}{2}\pmod{p},\\
                (85995)^2 \cdot 5 &\equiv \left(193596 \pm 39 \sqrt{421}\right)^2\pmod{p},\\
                       -504351432 &\equiv \pm 15100488\sqrt{421}\pmod{p},\\
                   (-504351432)^2 &\equiv (15100488)^2 \cdot 421\pmod{p},\\
               158371952330592000 &\equiv 0\pmod{p}.
    \end{align*}
    The only prime $p\not\in \mathcal{P}_2$ that divides 158371952330592000 is $p = 11$. The constant coefficients of $H_{-15}(X)$ and $f(X)$ are both multiples of~11, and it is now easy to verify that the only common root of these two polynomials modulo 11 is 0.
\end{proof}

With these ingredients, we are ready to explicitly describe the graph structure of~$\mathcal{S}_2^p$. As in Theorem \ref{thm:DG},
we consider three cases according to the structure of $\mathcal{G}_2(\Fp)$, covered in Theorems~\ref{thm:spinestructure_ell2_p1mod4}, \ref{thm:spinestructure_ell2_p3mod8}, and \ref{thm:spinestructure_ell2_p7mod8}, respectively. As before, we refer to the document \href{https://github.com/TahaHedayat/LUCANT-2025-Supersingular-Ell-Isogeny-Spine/blob/main/SmallCharacteristicGraphDescription.pdf}{\texttt{Small\-Characteristic\-Graph\-Description.pdf}} at \cite{LUCANTGitHub} for the primes~$p$ with $2 \le p \le 13$. 

\begin{theorem}[Spine structure, $p\equiv 1\pmod{4}$ and $\ell = 2$]\label{thm:spinestructure_ell2_p1mod4}
    Let $p \ge 17$ with $p\equiv 1\pmod{4}$.
    When mapping $\mathcal{G}_2(\Fp)$ into the spine $\mathcal{S}_2^p$, there may be stacking, folding, or attachment by new edges. The following congruence classes of $p$ determine precisely which of these occur and how often.
    \begin{enumerate}
            \item $p =29$: $\mathcal{G}_{2}(\mathbb{F}_{29})$ has two components: the component containing the two vertices with $j=8000$ folds, and there is an edge attachment connecting this folded component to the other component. The spine $\slp[2]$ is the entire 2-isogeny graph $\mathcal{G}_2(\Fpbar)$.
            \item $p \equiv 29, 101 \pmod{120}$, $p \neq 29$: the component containing the two vertices with $j=8000$ folds, all other components stack, and there is an edge attachment connecting two stacked components. The spine $\slp[2]$ consists of one isolated vertex corresponding to $j=8000$ with a loop, one component with four vertices, and the remaining $h(-4p)/2 - 5$ vertices are joined in pairs. 
            \item $p \equiv 41, 89 \pmod{120}$: all components stack, and there is an edge attachment. The spine $\slp[2]$ consists of one connected component with four vertices, and the remaining $h(-4p)/2 - 4$ vertices are joined in pairs.
            \item $p \equiv 13, 37, 53, 61, 77, 109 \pmod{120}$: the component containing the two vertices with $j=8000$ folds, all other components stack, and there are no edge attachments. The spine $\slp[2]$ consists of one isolated vertex corresponding to $j = 8000$ with a loop, and the rest of the $h(-4p)/2 -1$ vertices connect in pairs.
            \item $p \equiv 1, 17, 49, 73, 97, 113 \pmod{120}$: all components stack and there are no edge attachments. The spine $\slp[2]$ consists of $h(-4p)/2$ vertices joined in pairs.
        \end{enumerate}
The five cases are summarized in Table~\ref{tab:p1mod4}.

\newcommand{\minitab}[2][r]{\begin{tabular}{#1}#2\end{tabular}}
\begin{table}[h]
\centering
\begin{tabular}{l|ll|c|}
%\cline{2-4}
\hhline{~|---|}
& \multicolumn{2}{c|}{\cellcolor[HTML]{C0C0C0}Edge attachment}                    & \multicolumn{1}{c|}{\cellcolor[HTML]{C0C0C0}No edge attachment} \\ \hline

\multicolumn{1}{|l|}{\cellcolor[HTML]{C0C0C0}} & 
\multicolumn{2}{c|}{\multirow{2}{*}{$p\equiv 41,89\pmod{120}$}}  & 
    \multicolumn{1}{r|}{$p\equiv 1,17,49,73,97,$} \\
\multicolumn{1}{|l|}{\multirow{-2}{*}{\cellcolor[HTML]{C0C0C0}No fold}} & & & \multicolumn{1}{r|}{$113 \pmod{120}$} \\ \hline

\multicolumn{1}{|l|}{\cellcolor[HTML]{C0C0C0}}                                      & \multicolumn{1}{l|}{\cellcolor[HTML]{F0F0F0} w/ folded comp.} & \cellcolor[HTML]{F0F0F0} not w/ folded comp. &  \multicolumn{1}{r|}{\multirow{3}{*}{\minitab{$p\equiv 13,37,53,61,77,$ \\ $109\pmod{120}$}}} \\ \cline{2-3}
\multicolumn{1}{|l|}{\cellcolor[HTML]{C0C0C0}} & \multicolumn{1}{c|}{\multirow{2}{*}{$p = 29$}} & $p\equiv 29,101\pmod{120}$,  & \\ 
\multicolumn{1}{|l}{\multirow{-3}{*}{\cellcolor[HTML]{C0C0C0}One fold}} & \multicolumn{1}{|c|}{} & \multicolumn{1}{c|}{$p\neq 29$} & \multicolumn{1}{r|}{} \\ \hline
\end{tabular}
        \caption{Spine structure for $p\equiv1\pmod{4}$}
        \label{tab:p1mod4}
\end{table}
\end{theorem}

\begin{proof}
    The case $p = 29$ can be verified by directly computing $\mathcal{S}^{29}_2$. 

    We consider the cases for stacking and folding first, followed by those of edge attachment. The individual congruence conditions can be combined by the Chinese Remainder Theorem to provide the statements in this theorem.

    \noindent\textit{Folding and stacking:} By Proposition~\ref{prop:folding}, the only connected components which could possibly fold are those containing $j = 8000$ and $j = 1728$. Since $p \equiv 1\pmod{4}$, $j = 1728$ is not a supersingular $j$-invariant. The $j$-invariant 8000 is supersingular over $\Fp$ whenever $p\equiv 5,7\pmod{8}$. Combining this with our assumption that $p\equiv 1\pmod{4}$, the connected component with vertices having $j$-invariant 8000 will fold whenever $p\equiv5\pmod{8}$. The rest of the components necessarily stack.

    \noindent\textit{Edge attachment:} Every edge that appears in $\mathcal{G}_2(\Fpbar)$ but does not already belong to $\mathcal{G}_2(\Fp)$ results in a double edge by \cite[Lem.\ 3.14]{Arpin}. If a new edge is not attaching, this would result in a triple edge in the case of $p\equiv 1\pmod{4}$ due to the structure of $\mathcal{G}_2(\Fp)$. Since there are no triple edges for $p>13$ by Proposition~\ref{prop:TripleEdges}, every new edge must produce an edge attachment.
    
    By Proposition~\ref{prop:EApnot7mod8}, attachment by a new edge happens when $p\equiv 11,29,41,59$, $89,101\pmod{120}$. Combining this with our assumption that $p\equiv 1\pmod{4}$ yields the congruence classes $p\equiv 29,41\pmod{60}$.

    To ascertain whether or not this edge attaches to the folding component (corresponding to $j = 8000$), we observe that $j = 8000$ is supersingular only when $p\equiv 5\pmod{8}$ (and $p \equiv 1 \pmod{4})$, and 8000 is a root of $H_{15}(X)$ only for $p = 29$ (under the condition $p>13$). This explains the second column of Table~\ref{tab:p1mod4}. The third column is obtained by sorting the remaining congruence classes modulo 120 by their congruence class modulo $8$.
\end{proof}

\begin{theorem}[Spine structure, $p\equiv 3\pmod{8}$ and $\ell = 2$]\label{thm:spinestructure_ell2_p3mod8}
    Let $p \ge 17$ with $p\equiv 3\pmod{8}$.
   When mapping $\mathcal{G}_2(\Fp)$ into the spine $\slp[2]$, there may be stacking, folding, or attachment by new edges.  The connected component of $\mathcal{G}_2(\Fp)$ containing the two vertices with $j = 1728$ always folds. The following congruence classes of $p$ determine precisely which of these occur and how often.

\begin{enumerate}
            \item $p = 59$: the folded component gets edge attached to another component by an edge between two vertices on the floor.
            \item $p \equiv 11, 59 \pmod{120}$ and $p \neq 11, 59$: an edge attachment takes place between two stacked components with the attaching edge being incident to two vertices on the floor.
            \item $p \equiv 19, 43, 67, 83, 91, 107 \pmod{120}$: no edge attachment takes place.
        \end{enumerate}
The three cases are summarized in Table~\ref{tab:p3mod8}.
\begin{table}[h]
    \centering
    \begin{tabular}{|
>{\columncolor[HTML]{C0C0C0}}l |c|}
\hline
No edge attachment         & $p \equiv 19, 43, 67, 83, 91, 107 \pmod{120}$    \\ \hline
EA w/ folded comp.     & $p = 59$                                         \\ \hline
EA not w/ folded comp. & $p \equiv 11, 59 \pmod{120}$ and $p \neq 11, 59$ \\ \hline
\end{tabular}
    \caption{Spine structure for $p\equiv 3\pmod{8}$}
    \label{tab:p3mod8}
\end{table}
\end{theorem}

\begin{proof}
    The case $p = 59$ can be again be observed directly by computing $\mathcal{S}_2^{59}$. 

    As before, we consider stacking and folding first, followed by edge attachment. Chinese remaindering again produces the specified congruence classes for $p$. 

    \noindent\textit{Folding and stacking:} By Proposition~\ref{prop:folding}, the only connected components which could possibly fold are those containing $j = 8000$ or $j = 1728$. Since $p\equiv 3\pmod{8}$, $j = 8000$ is not a supersingular $j$-invariant, but $j = 1728$ is supersingular, and folding happens for the component containing the two vertices with $j = 1728$.

    \noindent\textit{Edge attachment:} By Proposition~\ref{prop:EApnot7mod8}, attachment by a new edge happens when $p\equiv 11,29,41,59$, $89,101\pmod{120}$. Combining this with $p\equiv 3\pmod{8}$ yields $p\equiv 11,59\pmod{120}$. 

    To ascertain when the new edge attaches to the folded component, note that this component consists precisely of curves with $j = 1728$ or $j = 287496$. This can be seen by taking a model for $j = 1728$ over $\mathbb{Z}$ and computing the possible $2$-isogenies. For $p > 13$, the $j$-invariant $j=1728$ is not a root of $H_{-15}(X)$, and  $j=287496$ is a root of $H_{-15}(X)$ precisely when $p = 59$. In the remaining cases, the edge attachment occurs between stacking components. This produces Table~\ref{tab:p3mod8}.
\end{proof}

\begin{theorem}[Spine structure, $p\equiv 7\pmod{8}$ and $\ell = 2$]\label{thm:spinestructure_ell2_p7mod8}
    Let $p \ge 17$ with $p\equiv 7\pmod{8}$.
    When mapping $\mathcal{G}_2(\Fp)$ into the spine $\slp[2]$, there may be stacking, folding, or attachment by new edges. Only the unique connected component of $\mathcal{G}_2(\Fp)$ containing $j=1728$ and $j = 8000$ folds. The following congruence classes of $p$ determine the presence of new edges.
\begin{enumerate}
        \item $p \equiv 71, 119 \pmod{120}$: there is a new double-edge in $\slp[2]$ which may or may not be an attachment.
        \item $p \equiv 7, 23, 31, 47, 79, 103 \pmod{120}$: edge attachment does not occur.
\end{enumerate}
\end{theorem}

\begin{proof}
We proceed as in the proofs of the previous two theorems. 
By Proposition~\ref{prop:folding}, the only components of $\mathcal{G}_2(\Fp)$ which could possibly fold are those containing $j = 8000$ and $j = 1728$. For $p \equiv 7\pmod{8}$, both these are supersingular, and are distinct for $p > 13$. By \cite[Ex.\ 3.8]{Arpin}, exactly one of the occurrences of 1728 in $\mathcal{G}_2(\Fp)$ lies on the surface of a volcano. Since $h(-p)$ is odd by Proposition \ref{prop:classno}, any volcano rim contains an odd number of vertices. In order for folding to occur, this rim must also contain two adjacent vertices with the same $j$-invariant, which is only possible for $j = 8000$. Hence the unique component containing $j = 1728$ and $j = 8000$ folds. 
 
A new edge is added when $p\equiv 11, 14\pmod{15}$ by Proposition~\ref{prop:EApnot7mod8}, but this may or may not be an edge attachment. Other than this edge, there are no new non-loop edges, so edge attachment cannot occur. Combining $p\equiv 11,14\pmod{15}$ with $p\equiv 7\pmod{8}$ gives $p\equiv 71,119\pmod{120}$.
\end{proof}

\begin{remark}
    Theorem~\ref{thm:spinestructure_ell2_p7mod8} shows that precisely one connected component of $\mathcal{G}_2(\Fp)$ folds for $p\equiv 7\pmod{8}$. The structure of this folding component tells us that the smallest positive integer $k$ such that there exists an $\Fp$-rational isogeny of degree~$2^k$ between an elliptic curve with $j = 8000$ and an elliptic curve with $j = 1728$ is $k =  (\operatorname{ord}(\mathfrak{l}_2) - 1)/2$, where $\operatorname{ord}(\mathfrak{l}_2)$ denotes the (necessarily odd) order of a prime ideal above 2 in the class group of $\mathbb{Q}(\sqrt{-p})$. 
\end{remark}

\begin{remark}
Whether or not the new edge of Theorem~\ref{thm:spinestructure_ell2_p7mod8} (1) produces an edge attachment depends entirely on whether the roots of $H_{-15}(X)$ belong to the same connected component of $\mathcal{G}_2(\Fp)$. When $p\equiv 7 \pmod{8}$, the volcano structure of this graph leaves too many possibilities. There is no single simple condition to establish the existence of an endomorphism etc., so we are not able to determine whether or not the new (double) edge is attaching. We provide examples for each case. 
\end{remark}  

\begin{example}[$p = 71$, no edge attachment]
For $p = 71$, edge attachment is in principle possible by Theorem~\ref{thm:spinestructure_ell2_p7mod8}; however, it does not occur. The graph $\mathcal{G}_2(\mathbb{F}_{71})$ consists of a single component. It has $7$ vertices on the surface joined to $7$ vertices on the floor. This component folds and there is a new edge, but clearly this cannot be an edge attachment: any time $\mathcal{G}_2(\Fp)$ consists of a single connected component, edge attachment is not possible. See Figure~\ref{fig:p71ell2}.
\end{example}

\begin{figure}[h]
    \centering
    \begin{subfigure}[b]{0.45\textwidth}
        \centering
        \includegraphics[width=\linewidth]{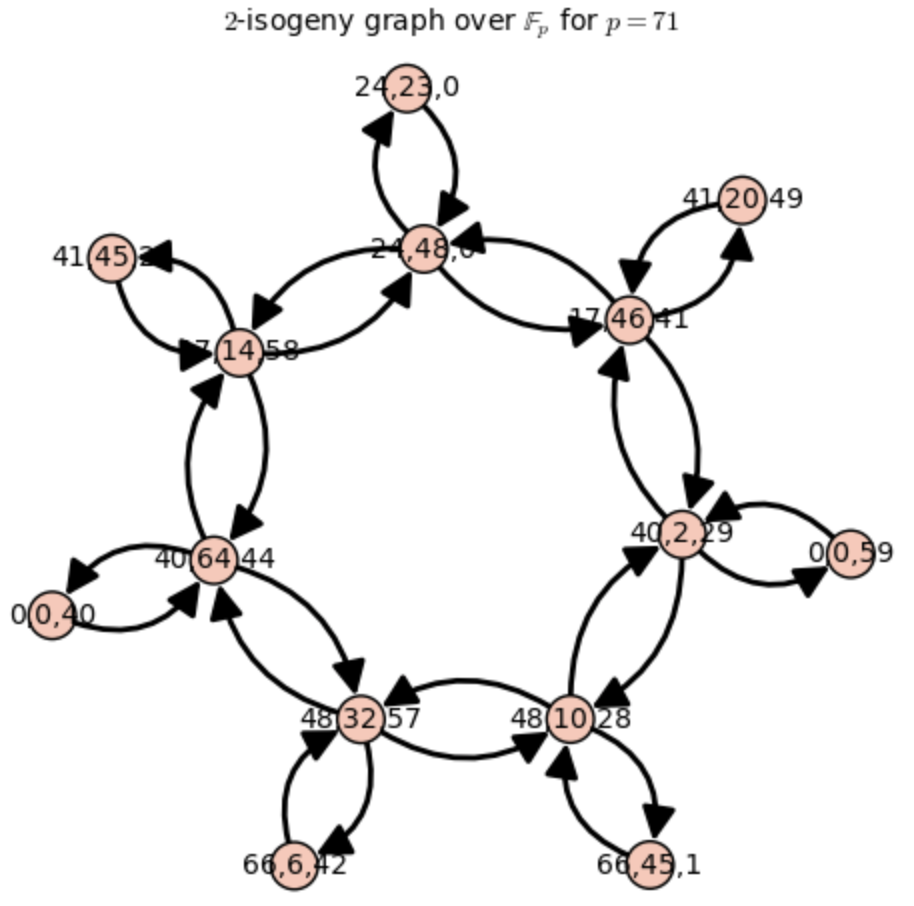}
        \caption{}
    \end{subfigure}\hfill
    \begin{subfigure}[b]{0.45\textwidth}
        \centering
        \includegraphics[width=\linewidth]{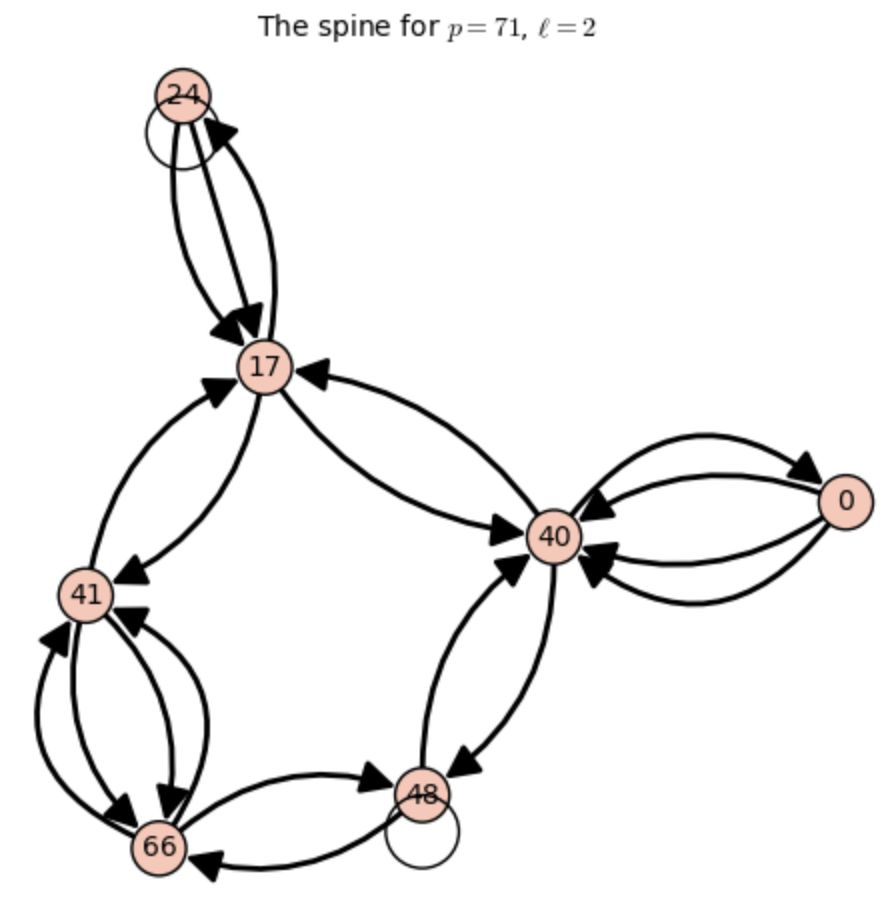}
        \caption{}
    \end{subfigure}
    \caption{(A): The graph $\mathcal{G}_2(\mathbb{F}_{71})$, 
    with vertices labeled by the triple of invariants $(j,c_4,c_6)$. (B): The spine graph $\mathcal{S}_2^{71}$,
    with vertices labeled by $j$-invariant. Images created in \cite{sage}.}
    \label{fig:p71ell2}
\end{figure}

\begin{example}[$p = 1319$, edge attachment]
For $p = 1319\equiv 71\pmod{120}$, edge attachment is possible by Theorem~\ref{thm:spinestructure_ell2_p7mod8}, and in fact it occurs. The graph $\mathcal{G}_2(\mathbb{F}_{1319})$ has five connected components. Each component is a volcano with 9 vertices on the surface and 9 vertices on the floor. Two pairs of connected components stack and the remaining connected component folds. The two stacked components attach at a new edge at the vertices with $j$-invariants 446 and 1103. See Figure~\ref{fig:enter-label}.
\end{example}

\begin{figure}
    \centering
    \begin{subfigure}[b]{0.355\textwidth}
        \centering
        \includegraphics[width=\linewidth]{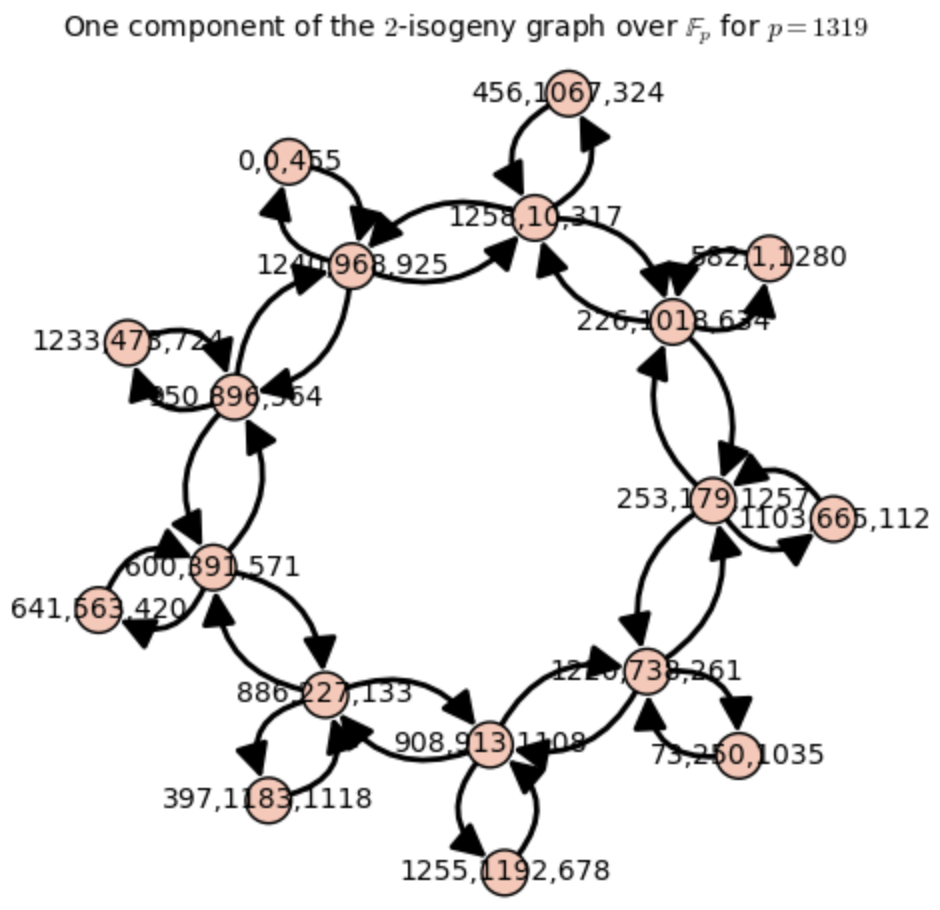}
        \caption{}
    \end{subfigure}\hfill
    \begin{subfigure}[b]{0.545\textwidth}
        \centering
        \includegraphics[width=\linewidth]{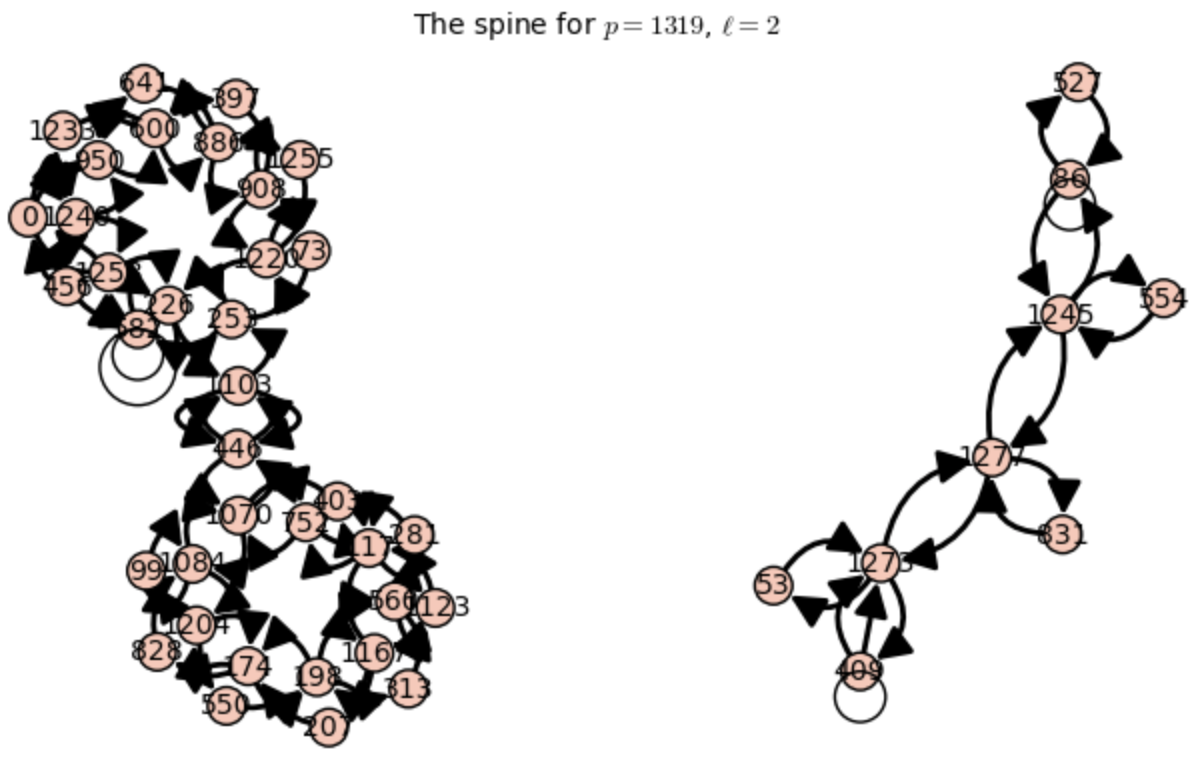}
        \caption{}
    \end{subfigure}
    \caption{(A): One (of five) connected components of $\mathcal{G}_2(\mathbb{F}_{1319})$,
    with vertices labeled by the triple of invariants $(j,c_4,c_6)$. (B): The spine graph $\mathcal{S}_2^{1310}$,
    with vertices labeled by $j$-invariant. Images created in \cite{sage}.}
    \label{fig:enter-label}
\end{figure}

\section{Structure of $\mathcal{S}_\ell^p$ for $\ell \ge 3$} \label{sec:ell3structurethms}

The process of moving from $\mathcal{G}_\ell(\Fp)$ to $\mathcal{S}_\ell^p$ for $\ell \ge 3$ is less involved than the analogous procedure for $\ell = 2$ due to the substantially simpler structure of $\mathcal{G}_\ell(\Fp)$ as a collection of disjoint cycles. We provide an overview of the general method for computing spine structures, with particular focus on the case $\ell = 3$, where we describe the structure of $\mathcal{S}_3^p$ in a manner similar to Theorems \ref{thm:spinestructure_ell2_p1mod4}-\ref{thm:spinestructure_ell2_p7mod8}. 

Note that if $(\frac{-p}{\ell}) = -1$, then $\gf$ has no edges, including loops, by Theorem~\ref{thm:DG} and Corollary~\ref{cor:loops23}. In this case, all components (i.e.\ isolated vertices) stack, none fold, there is no vertex attachment, and only new edges are introduced.

\smallskip

\noindent \textbf{Step 1: Loops.} Determine the vertices in $\gfb$ belonging to $\Fp$ that incur loops by factoring $\Phi_\ell(X,X)$ over $\Fp$ and determining the congruence classes for $p$ such that these vertices correspond to supersingular $j$-invariants in $\Fp$. For edge attachment (investigated in step 3), also ascertain which of these loops belong to $\gf$ via Proposition \ref{prop:loops}.

By Corollary \ref{cor:loops23}, $\mathcal{G}_3(\Fp)$ contains loops loops only for $\ell = 11$. The loops in $\mathcal{G}_3(\Fpbar)$ are given as follows. 

\begin{lemma}[Loops in $\mathcal{G}_3(\Fpbar)$]\label{lem:loops3}
Let $p \ne 2, 3, 11$. Loops occur in $\mathcal{G}_3(\Fpbar)$ at vertices corresponding to precisely the following $j$-invariants, all belonging to $\Fp$:
\begin{align*}
    0 \text{ and } 54000 &\text{ if }p \equiv 2 \pmod{3} \\
    8000 & \text{ if }p\equiv 5,7\pmod{8}\\
    -32768 &\text{ if }p\equiv 2, 3, 6, 7, 8, 10 \pmod{11}.
\end{align*}
\end{lemma}
\begin{proof}
  We have 
  \[  \Phi_3(X,X) = -X(X - 54000)(X - 8000)^2(X + 32768)^2 = H_{-3} H_{-12} H_{-8}^2(X) H_{-11}(X)^2. \]
   The congruence conditions on $p$ come from the inertness of $p$ in the corresponding quadratic fields. 
\end{proof}

The $j$-invariants of Lemma~\ref{lem:loops3} are not all distinct for $p \le 5$. 

\smallskip

\noindent \textbf{Step 2: Folding and vertex attachment.} If $p \equiv 3 \pmod{4}$, then there are two different components containing 1728 by Theorem \ref{thm:3.18A+}. The two components fold and get attached at vertex $j = 1728$, and all other components stack.

If $p \equiv 1 \pmod{4}$, then the components that can fold are those containing two neighbours in $\gf$ with the same $j$-invariant, where $\gf$ does not incur a loop (but of course $\gfb$ does). These vertices, along with loops (see step 1), correspond to vertices $j \in \Fp$ such that $\Phi_\ell(j, j) = 0$. 

\begin{proposition}[Folding and vertex attachment for $\ell = 3$] \label{prop:folding3}
If $p \equiv 11 \pmod{12}$, then only the two components containing $j = 1728$ fold and attach at 1728. If $p \equiv 5 \pmod{12}$, then only the component containing $j = 0$ folds and there is no vertex attachment. Else there is neither folding nor vertex attachment. 
\end{proposition}
\begin{proof}
For $p \equiv 11 \pmod{12}$, this is Theorem \ref{thm:3.18A+}. Suppose $p \equiv 5 \pmod{12}$. Then $j = 0, 54000$ are supersingular (they are the roots of $H_{-3}(X)$ and $H_{-12}(X)$, respectively), whereas 1728 is not. Two elliptic curves defined over $\Fp$ with $j = 0$ are $E_0: y^2 = x^3 + 1$ and $E_0^t: y^2 = x^3 - 3$, its twist by $-3$. Since $-3 \notin \Fp$, the are non-isomorphic over $\Fp$. There is a 3-isogeny from $E_0$ to $E_p^t$ with kernel $\langle (0,1) \rangle$ and hence defined over $\Fp$, i.e.\ corresponding to an edge in $\mathcal{G}_3(\Fp)$. 

Now 0 and 54000 are the only $\Fp$-vertices $j$ in $\mathcal{G}_3(\Fpbar)$ that are roots of $\Phi_3(j,Y)$. Using the same reasoning as in the proof of \cite[Thm.\ 3.18]{Arpin}, we can show that all their occurrences in $\mathcal{G}_3(\Fp)$ belong to the same component, and this is the only component that folds. 

For all other primes $p$, we have $(\frac{-p}{3}) = -1$, so $\mathcal{G}_3(\Fp)$ contains no edges.
\end{proof}

\noindent \textbf{Step 3: New edges and edge attachment.} From the factorization of the polynomials Res$_\ell(X)$ into Hilbert class polynomials, determine the new edges; it may not always be possible to ascertain whether or not they attach. 

We sketch the idea for $\ell = 3$. We have 
\begin{equation} 
\mbox{Res}_3(X) = -3^3X^2(X-8000)^2(X-1728)^2 H_{-20}(X) H_{-32}(X)H_{-35}(X). \label{eq:Res03} 
\end{equation}
We first find congruence conditions on $p$ under which $H_{-20}(X)$, $H_{-32}(X)$, $H_{-35}(X)$ have supersingular roots in $p$, and when any two or all three of these polynomials share a root. Next, rather than resorting to another resultant polynomial to check for triple edges as we did in Proposition \ref{prop:TripleEdges}, we investigate edge incidence at each of the loop vertices $j$ listed in Lemma \ref{lem:loops3} by considering the polynomial $\Phi_3(j,Y)$. For each of these vertices, we check whether there are new loops at $j$ or a new edge from $j$ to 0, 1728 or another vertex (resulting in a multi-edge). We have 
\begin{align*}
\Phi_3(0,Y) &= Y(Y+12288000)^3 = Y H_{-27}(Y)^3, \\
\Phi_3(54000,Y) &= H_{-12}(Y) H_{-108}(Y) , \\
\Phi_3(8000, Y) &= (Y-8000)^2 H_{-72}(Y) , \\
\Phi_3(-32768,Y) &= H_{-11}(Y)^2 H_{-99}(Y) , \\
\Phi_3(1728,Y) &= H_{-36}(Y)^2.
\end{align*}
For each of these polynomials, we check whether the factors have roots in $\Fp$ that represent supersingular $j$-invariants, and whether roots between factors collide modulo $p$. All these conditions impose congruence restrictions on $p$. Putting it all together, we obtain three spine structure theorems, differentiated by number of folding components. For brevity, put 
\[  \mathcal{P}_3 = \{ 5, 7, 11, 13, 17, 19, 23, 29, 31, 41, 47, 59, 61, 71, 79, 89, 101, 139, 151, 199, 271 \}. \]
The spines for $p \in \mathcal{P}_3$, along with the graphs $\gf$ and $\gfb$, are explicitly described in the document \href{https://github.com/TahaHedayat/LUCANT-2025-Supersingular-Ell-Isogeny-Spine/blob/main/SmallCharacteristicGraphDescription.pdf}{\texttt{Small\-Characteristic\-Graph\-Description.pdf}} at~\cite{LUCANTGitHub}, so we only consider primes $p \notin \mathcal{P}_3$ here.

\begin{theorem}[Spine structure, $\ell = 3$, no folding] \label{thm:spinestructure_ell3-nofold}
Suppose $p \notin \mathcal{P}_3$. In the following cases, no connected component of $\mathcal{G}_3(\Fp)$ folds and no vertex attachment takes place. In addition to new loops, new edges are added as follows.
\begin{enumerate}
\item None when 
\begin{align*}
    p \equiv \ & 1, 13, 37, 43, 67, 73, 97, 109, 121, 157, 163, 169, 187, 193,  \\ & 253, 277, 283, 289, 307, 313, 337, 361, 373, 397, 403, 421, \\ & 433, 457, 493, 517, 523, 529, 541, 547, 577, 589, 613, 643, \\&  667, 673, 697, 709, 733, 757, 781, 787, 793, 817 \pmod{840}
\end{align*}
\item  One when 
\begin{align*}
    p \equiv \ & 61, 103, 127, 181, 211, 223, 229, 241, 247, 331, 349, 367, \\ &  379, 409, 463, 481, 487, 499, 571, 583, 601, 607, 649, 661, \\ & 703, 727, 739, 769, 823, 829 \pmod{840}.
\end{align*}
\item Two that do no share any vertices when
\[ p \equiv 19, 79, 139, 151, 319, 451, 619, 631, 691, 751, 799, 811 \pmod{840} \]
 \item Three that do no share any vertices when
  $$p \equiv 31, 199, 271, 391, 439, 559 \pmod{840}.$$       
\end{enumerate}
\end{theorem} 

\begin{theorem}[Spine structure, $\ell = 3$, one component folds] \label{thm:spinestructure_ell3-onefold}
Suppose $p \notin \mathcal{P}_3$. In the following cases, the connected component of $\mathcal{G}_3(\Fp)$ containing $j=0$ folds and no vertex attachment takes place. In addition to new loops, new edges are added as follows.
\begin{enumerate}
\item None when 
\begin{align*}
    p \equiv \ &  17, 29, 53, 113, 137, 149, 173, 197, 221, 233,257, 281, 293, 317, \\ & 353, 377, 389, 401, 437, 449, 473, 533, 557, 569, 593, 617, 641, \\ & 653, 677, 701, 713, 737, 773, 797, 809, 821 \pmod{840} .
 \end{align*}      
 \item One when   
\[ p \equiv 41, 89, 101, 209, 269, 341, 461, 509, 521, 629, 689, 761 \pmod{840} \]
\end{enumerate}
\end{theorem} 

\begin{theorem}[Spine structure, $\ell = 3$, two components folds] \label{thm:spinestructure_ell3-twofold}
Suppose $p \notin \mathcal{P}_3$. In the following cases, the two connected components of $\mathcal{G}_3(\Fp)$ containing $j=1728$ fold and get attached at $j = 1728$. In addition to new loops, new edges are added as follows.
\begin{enumerate}
\item None when
\[ p \equiv  83, 107, 227, 323, 347, 443, 467, 563,  587, 683, 803, 827 \pmod{840}. \]
\item One when 
\begin{align*}
p \equiv \ & 11, 23, 47, 143, 167, 179, 263, 383, 407, 491, 503, \\ &  527, 611, 647, 659, 743, 767, 779 \pmod{840}.
\end{align*}
\item Two that do no share any vertices when
\[ p \equiv 59, 71, 131, 191, 239, 251, 299, 359, 419, 431, 599, 731 \pmod{840}. \]
\item Three that do no share any vertices when
\[ p \equiv 311, 479, 551, 671, 719, 839 \pmod{840} . \]
\end{enumerate}  
\end{theorem}

Note that Theorem \ref{thm:spinestructure_ell3-nofold} covers exactly the setting when $\mathcal{G}_\ell(\Fp)$ has no edges. Theorem \ref{thm:spinestructure_ell3-onefold} deals with the case when $j = 0$ is supersingular and $j = 1728$ is not, while in Theorem \ref{thm:spinestructure_ell3-twofold}, both $j = 0$ and $j = 1728$ are supersingular.

One obstacle to obtaining general explicit structure results about $\slp$ for $p \ge 5$, stated only in terms of congruence classes of $p$, is the fact that degrees of Hilbert class polynomials grow as the corresponding discriminant increases in absolute value. Already for $\ell = 5$, this becomes a problem: the polynomial Res$_5(X)$ has degree 54 and decomposes over $\ZZ$ into linear, quadratic and quartic irreducible factors. While formulas exist for the roots of degree 4 polynomials, the conditions on $p$ become increasingly complicated, and for irreducible factors of degree and higher, such root formulas may no longer exist. 

\section{Experiments on the Structural Properties of Isogeny Graphs} \label{sec:exp}

As mentioned in the introduction, $\gfb$ is an optimal expander graph and in fact a Ramanujan graph for $p \equiv 1 \pmod{12}$ (when neither 0 nor 1728 is supersingular). The fact that we can partition the vertices of $\gfb$ into two categories, namely the $\Fp$-vertices and the $(\mathbb{F}_{p^2} \setminus\Fp)$-vertices,  casts doubt upon the assumption that $\gfb$ behaves like a random graph. To shed further light on this question, we gathered a substantial amount of data on graph-theoretic invariants of spines~$\slp$ for $\ell \le 5$ and many primes $p$. We provide an in-depth study of a selection of these invariants in this section. 

\subsection{Relevant notions and definitions} \label{ssec:defs}

Let $G$ be a directed graph with vertex set $V(G)$ and edge set $E(G)$. Recall that $G$ is said to be \emph{strongly connected} if it contains a directed path between any two vertices. The \emph{distance} $d(v,w)$ from a vertex $v$ to a vertex $w$ is the length (i.e.\ the number of edges) in a shortest path from $v$ to $w$ in $G$. If no such path exists, we set $d(v,w)=\infty$. Also, $d(v,v)=0$. 

\begin{definition}[Eccentricity]\label{def:eccentricity}
    Let $v \in V(G)$. The \textit{out-eccentricity} 
    of $v$ 
    is the quantity
    \begin{align*}
    ecc^+(v) &= \max \set{d\brac{v,w} : w \in V(G) \text{ and } d\brac{v,w} \neq \infty}, \\
    \end{align*}
\end{definition}

This notion captures the furthest distance required to travel from $v$ to any other vertex of $G$.

\begin{definition}[Diameter]\label{def:diameter}
If every component of $G$ is strongly connected, then the \emph{diameter} of $G$ is the quantity
        $$diam(G) = \max\set{ecc^+(v): v \in V(G)}.$$
Otherwise, the diameter of $G$ is infinite, i.e.\ $diam(G) = \infty$.
\end{definition}

The diameter is the largest distance between any two vertices of $G$. Similarly, the radius is the smallest distance between any two vertices of $G$.

\begin{definition}[Radius]\label{def:radius}
If every component of $G$ is strongly connected, then the \emph{radius} of $G$ is the quantity
        $$rad(G) = \min\set{ecc^+(v): v \in V(G)}.$$
Otherwise, the radius of $G$ is infinite, i.e.\ $rad(G) = \infty$.
\end{definition}

Finally, the center of $G$ is the set of vertices of $G$ for which the distance to any other vertex is minimal (i.e.\ takes on the value of the radius). 

\begin{definition}[Center]\label{def:center}
    If the radius of $G$ exists, then the \textit{center} of $G$ is the set 
    $$cen(G) = \set{v \in V(G) : ecc^+(v) = rad(G)}.$$
\end{definition}

As suggested by the name, center vertices can be thought of ``central'' in the sense that they are better connected to the entire graph compared to vertices outside the center.

\subsection{Diameter of the spine $\mathcal{S}_2^p$} \label{ssec:diam}
In light of the structure theorems from Section~\ref{sec:ell2structurethms}, the diameters of the connected components of $\slp[2]$ can be explicitly computed in almost all cases. Again, for this entire section assume $p > 13$.  
\begin{remark}[Diameters for $\ell>2$]
    One could use the structure theorems from Section ~\ref{sec:ell3structurethms} to make similar statements about the diameter of $\slp[3]$ (or even $\slp$ for $p \ge 5$). However, compared to the case $\ell = 2$, less is known about the order $r$ of an ideal class of a prime ideal above an odd prime $\ell$ in the appropriate class group, and it is harder to obtain concrete statements. Indeed, the $\ell>2$ case is similar to the $p\equiv 7\pmod{8}$ case for $\ell = 2$ (Theorem~\ref{thm:spinediameters_p7mod8}), where the lengths of the cycles in $\gfb$ depend on the quantity~$r$.
\end{remark}

Following the structure theorems in Section~\ref{sec:ell2structurethms}, we determine the diameters of the connected components of $\mathcal{S}_2^p$ in the cases $p\equiv 1\pmod{4}$ and $p\equiv 3\pmod{8}$. Examples illustrating each case can be found in a notebook at~\cite{LUCANTGitHub} entitled \href{https://github.com/TahaHedayat/LUCANT-2025-Supersingular-Ell-Isogeny-Spine/blob/main/SpineDiameter_examples.ipynb}{\texttt{SpineDiameter\_examples.ipynb}}. For $p\equiv 7\pmod{8}$, we cannot determine edge attachments, which prevents us from classifying the diameter of $\mathcal{S}_2^p$ completely in this case.  

Recall that (the undirected version of) $\mathcal{G}_2(\Fp)$ consists of pairs of vertices joined by a single edge when $p\equiv1\pmod{4}$.
Mapping into the spine $\mathcal{S}_2^p$, the generic behavior for these components is to stack. The diameters of the connected components of $\slp[2]$ thus depend on the number of folds and edge attachments. 

\begin{theorem}[Spine Diameters, $p\equiv 1\pmod{4}$ and $\ell = 2$] \label{thm:spinediameters_p1mod4}
Let $p \ge 17$ with $p\equiv 1\pmod{4}$. There are $h(-4p)/2$ vertices in $\mathcal{S}_2^p$. The following congruence conditions on $p$ completely determine the diameters of the components of $\mathcal{S}_2^p$:
\begin{enumerate}
    \item If $p\equiv 1,17,49,73,97,113\pmod{120}$, then $\mathcal{S}_2^p$ consists of $h(-4p)/2$ vertices joined in pairs, so each connected component of $\mathcal{S}_2^p$ has diameter~1.

    \item If $p\equiv 41,89\pmod{120}$, then one connected component of $\mathcal{S}_2^p$ has four vertices (diameter 3), and the remaining $h(-4p)/2 - 4$ vertices are joined in pairs (diameter~1).

    \item If $p=29$, then $\mathcal{S}_2^p = \mathcal{G}_2(\Fpbar)$, so it contains three vertices and the diameter is~2.

    \item If $p\equiv 29, 101\pmod{120}$, then one connected component of $\mathcal{S}_2^p$ has a single vertex with a loop, one connected component has four vertices (diameter 3), and the remaining $h(-4p)/2 - 5$ vertices are joined in pairs (diameter~1).

    \item If $p\equiv 41,89\pmod{120}$, then one connected component of $\mathcal{S}_2^p$ has four vertices (diameter 3) and the remaining $h(-4p)/2-4$ vertices are joined in pairs (diameter~1).

    \item If $p\equiv 13,37,53,61,77,109\pmod{120}$, then one connected component of $\mathcal{S}_2^p$ is a single vertex with a loop, and the remaining $h(-4p)/2 - 1$ vertices are joined in pairs (diameter~1).
\end{enumerate}
\end{theorem}
\begin{proof}
    This is a direct result of applying Theorem~\ref{thm:spinestructure_ell2_p1mod4} to the possible graph structure of $\mathcal{G}_2(\Fp)$ given in Theorem~\ref{thm:DG}.
\end{proof}

\begin{theorem}[Spine Diameters, $p\equiv 3\pmod{8}$ and $\ell = 2$] \label{thm:spinediameters_p3mod8}
Let $p \ge 17$ with $p\equiv 3\pmod{8}$. There are $2h(-p)$ vertices in $\mathcal{S}_2^p$. The following congruence conditions on $p$ completely determine the diameters of the components of $\mathcal{S}_2^p$:
\begin{enumerate}
    \item If $p\equiv 19, 43, 67, 83, 91, 107\pmod{120}$, then one connected component of $\mathcal{S}_2^p$ consists of two adjacent vertices (diameter 1) and the remaining $2h(-p) - 2$ vertices are joined in groups of four in tripods from case (2) of \ref{thm:DG} (diameter 2).

    \item If $p = 59$, then $\mathcal{S}_2^p$ is a single connected component formed by a tripod edge joined by an edge to a folded tripod component (diameter 4).

    \item If $p\equiv 11,59\pmod{120}$ ($p\neq 59$), then one connected component of $\mathcal{S}_2^p$ consists of two adjacent vertices (diameter 1), one connected component is 8 vertices in two tripod shapes joined by a double edge (diameter 5), and the remaining $2h(-p) - 9$ vertices are adjacent in groups of four in tripod formation.
\end{enumerate}
\end{theorem}

\begin{proof}
    Follows directly from applying Theorem~\ref{thm:spinestructure_ell2_p3mod8} to Theorem~\ref{thm:DG}.
\end{proof}

\begin{theorem}[Spine Diameters, $p\equiv 7\pmod{8}$ and $\ell = 2$] \label{thm:spinediameters_p7mod8}
    Let $p \ge 17$ with $p\equiv 7\pmod{8}$. There are $h(-p)$ vertices in $\mathcal{S}_2^p$. Let $r$ denote the order of the ideal class generated by either of the prime ideals above $2$ in the class group of $\mathbb{Q}(\sqrt{-p})$. 
    If $p\equiv 7, 23, 31, 47, 79, 103\pmod{120}$, then the diameter of the spine is $(r+3)/2$.

\end{theorem}

\begin{proof}
    Follows directly from applying Theorem~\ref{thm:spinestructure_ell2_p7mod8} to Theorem~\ref{thm:DG}.
\end{proof}

If $p\equiv 71, 119\pmod{120}$, then the diameter of $\mathcal{S}_2^p$ is uncertain. Attaching edges will approximately double the diameter of the resulting connected component, but there are no clear congruence conditions for when this occurs. See Figure~\ref{fig:p7mod8} for a plot visualizing the mean diameters of the components of $\slp[2]$ for a range of primes $p\equiv 7\pmod{8}$. The notebook used to collect the data in this figure can be found in the file \href{https://github.com/TahaHedayat/LUCANT-2025-Supersingular-Ell-Isogeny-Spine/blob/main/SpineDiameter.ipynb}{\texttt{SpineDiameter.ipynb}} at~\cite{LUCANTGitHub}.

\begin{figure}[h]
    \centering
    \includegraphics[width=0.7\linewidth]{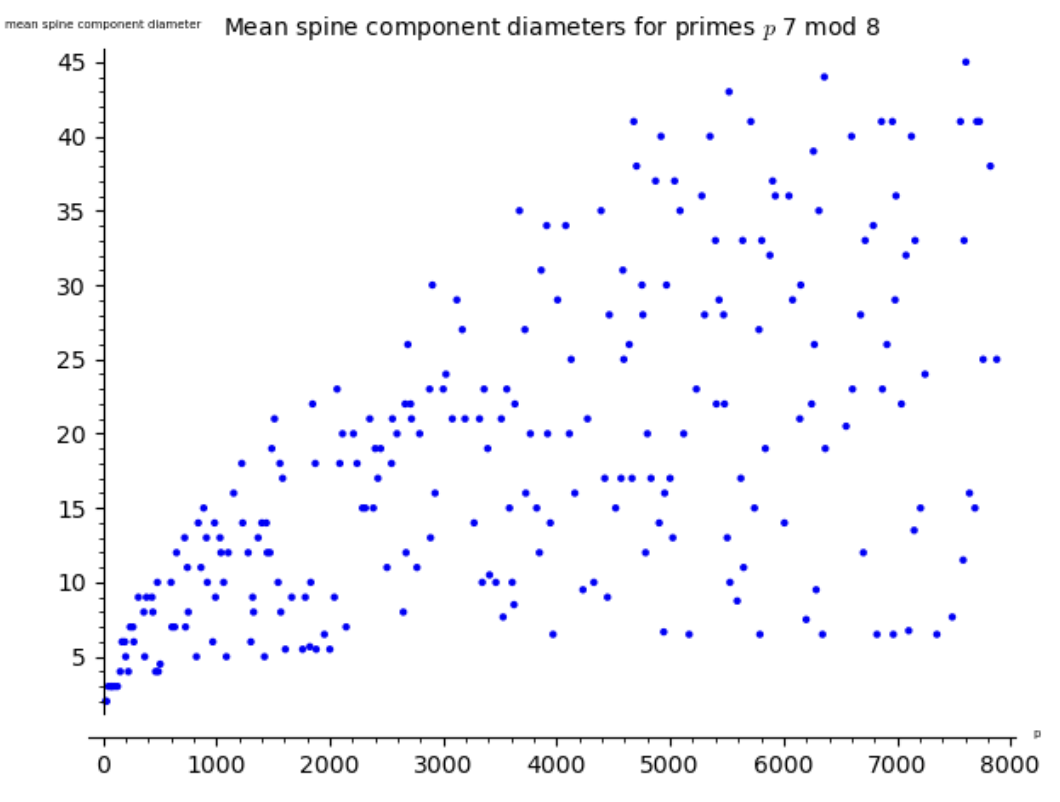}
    \caption{Mean spine component diameters in $\mathcal{G}_2(\Fpbar)$, for 250 primes $p \equiv 7\pmod{8}$ with $23 \le p \le 7879$.}
    \label{fig:p7mod8}
\end{figure}

\subsection{Center of $\gfb$} \label{ssec:center}

In this work, we computed the centers of supersingular elliptic curve $2$- and $3$-isogeny graphs and counted the number of center vertices belonging to their respective spines. The accompanying data are listed in the file \href{https://github.com/TahaHedayat/LUCANT-2025-Supersingular-Ell-Isogeny-Spine/blob/main/center012925.csv}{\texttt{center012925.csv}}. They were generated with the notebook \href{https://github.com/TahaHedayat/LUCANT-2025-Supersingular-Ell-Isogeny-Spine/blob/main/Center_DataGeneration.ipynb}{\texttt{Center\_Data\-Generation.ipynb}} and plotted using \href{https://github.com/TahaHedayat/LUCANT-2025-Supersingular-Ell-Isogeny-Spine/blob/main/Center_Data\-Processing.ipynb}{\texttt{Center\_DataProcessing.ipynb}}. All these sources can be found at \cite{LUCANTGitHub}.

Recall that the center of a graph (Definition~\ref{def:center}) is the set of vertices with minimal out-eccentricity. These vertices are well-connected to every other vertex in the graph. Considering that the $p$-power Frobenius map is a graph automorphism on $\gfb$ that fixes the vertices of $\slp$, one might expect the spine vertices to be over-represented in the center of $\gfb$. To see this, we note that for any vertex $v \in \Fp$, the set of distances from $v$ displays a symmetry whereby distances can be paired up. Specifically, if $w$ is any vertex of $\gfb$ and $w^p$ its Frobenius conjugate, then $d(v,w) = d(v,w^p)$. This property does not hold for vertices $v \in \mathbb{F}_{p^2} \setminus \Fp$. So the set of distances from vertices in $\Fp$ only supports ``half the randomness'' of the set of distances from vertices outside $\Fp$.

In our first experiment, we counted the number of $\Fp$-vertices in the center of~$\mathcal{G}_2(\Fpbar)$ as $p$ ranges through the  2260 primes from $5$ to $19997$. Immediately, a striking wave-like pattern emerged. In order to ascertain whether this behaviour was particular to the $\Fp$-vertices in the center or the shadow of a broader phenomenon, we repeated this experiment for the entire center of $\mathcal{G}_2(\Fpbar)$ using the same range of primes. The results are plotted in blue in Figure~\ref{fig:center-size}. The wave shapes are more pronounced over $\Fpbar$ compared to~$\Fp$; analogous results for $\ell=3$ look similar. 

We are indebted to Jonathan Love for the following explanation. The observed wave pattern in Figure~\ref{fig:center-size} closely matches the expected behavior of the minimum value of an integer-valued distribution with slow growing mean -- in this case, the out-eccentricity whose minimum value (which is the radius) grows approximately as $\log(p/12)$ -- and very small standard deviation. Each wave corresponds to primes~$p$ for which $\mathcal{G}_2(\Fpbar)$ has a fixed radius $r$. As $p$ grows, so does $p/12$ (the number of vertices in $\mathcal{G}_2(\Fpbar)$), allowing fewer and fewer vertices with out-eccentricity $r$ until no more such vertices exist. At this point, the radius of $\mathcal{G}_2(\Fpbar)$ jumps to $r+1$, an out-eccentricity value taken on by many more vertices, which starts the next wave. Our experiments confirm that the radius of $\mathcal{G}_2(\Fpbar)$ is generally only slightly larger than $\log_2(p/12)$ in the range of primes under investigation. Thus, if a wave corresponding to a given radius $r$ peaks at a prime $p$, then the peak of the next wave, corresponding to radius $r+1$, should be located close to $2p$; in other words, the distance between consecutive wave crests doubles each time. Our data bears this out as well.

The green `Discrete Gaussian' points in Figure~\ref{fig:center-size} were obtained via a Discrete Gaussian sampler as follows. 
For any prime $p$ with $p \equiv 1 \pmod{12}$, simulate a 3-regular graph $G$ with $(p-1)/12$ vertices. Assign an out-eccentricity to any vertex of $G$ by sampling from a normal distribution with mean $1.8\log(p)$ and standard deviation $0.38$ and plot the floor function of the sampled values. Although out-eccentricities of adjacent vertices in $\mathcal{G}_2(\Fpbar)$ are in actuality not independent, it is evident that this model matches our observed center size data quite closely.

Additionally, we thank Thomas Decru and Jonathan Komada Eriksen for the observation that the likelihood of a vertex $v$ of $\mathcal{G}_\ell(\Fpbar)$ to belong to the center can be estimated by the discrepancy between the theoretically possible and the actual number of ways in which it can achieve $ecc^+(v) = r$, where $r$ is the radius of $\mathcal{G}_\ell(\Fpbar)$. Let $T_\ell$ be the tree rooted at~$v$ where~$v$ has $\ell+1$ children, all other interior vertices have $\ell$ children, and all leaf nodes are located at level~$r$. Then~$T_\ell$ models an idealized version of the possible paths of length at most $r$ from~$v$ in~$\mathcal{G}_\ell(\Fpbar)$, assuming no cycles are encountered. The difference $\epsilon = |V(T_p)| - |(\mathcal{G}_\ell(\Fpbar))|$ is a measure of the likelihood for a vertex of $\mathcal{G}_\ell(\Fpbar)$ to lie in the center. For $\ell = 2$, the number of vertices in $T_2$ is $|V(T_2)| = 1 + 3(2^{r} - 1)$. The quantity $\epsilon = |V(T_2)| - |V(\mathcal{G}_2(\Fpbar))|$, scaled by $1/12$ for best fit, is plotted in red in Figure~\ref{fig:center-size}.

\begin{figure}[ht]
    \centering
    \includegraphics[width=0.8\linewidth]{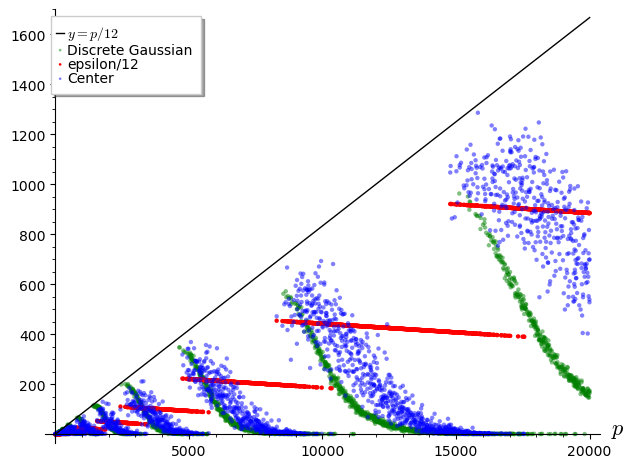}
    \caption{Size of the center of $\mathcal{G}_2(\Fpbar)$ (in blue) and discrete Gauss sampling with $\mu = 1.8\log(p)$, $\sigma = 0.38$, (in green). The red segments represent the estimate $\epsilon/12$ for a vertex of $\mathcal{G}(\Fpbar)$ to belong to the center, and the line $y = p/12$ (in black) approximates the total number of vertices in $\mathcal{G}_2(\Fpbar)$.}
    \label{fig:center-size}
\end{figure}

To ascertain if elliptic curves with extra automorphisms had any effect on the size of the center of $\mathcal{G}_2(\mathbb{F}_2)$, we separated our data into congruence classes of $p\pmod{4}$ and $p\pmod{3}$, and the resulting plots are found in Figure~\ref{fig:FpCenterSorted}. 
No definitive pattern emerges for the congruence classes of $p\pmod{3}$. While the data points for both congruences classes of $p$ modulo~4 are spread out over the entire data range, higher center size counts (i.e.\ more data points in the wave peaks) appear for $p\equiv 3\pmod{4}$. This is to be expected because the radius $r_p$ of $\mathcal{G}_\ell(\Fpbar)$ tends to be larger in this case and hence easier to attain. In particular, the vertex associated to $j=1728$ has only one neighbor distinct from itself, wheres a generic vertex of~$\mathcal{G}_\ell(\Fpbar)$ is expected to have three neighbors. A larger  radius~$r_p$ makes it easier for a random vertex to achieve out-eccentricity at most~$r_p$  and thus belong to the center of $\mathcal{G}_\ell(\Fpbar)$.

\begin{figure}[h]
     \centering
     \begin{subfigure}[b]{0.53\textwidth}
         \centering
         \includegraphics[width=\textwidth]{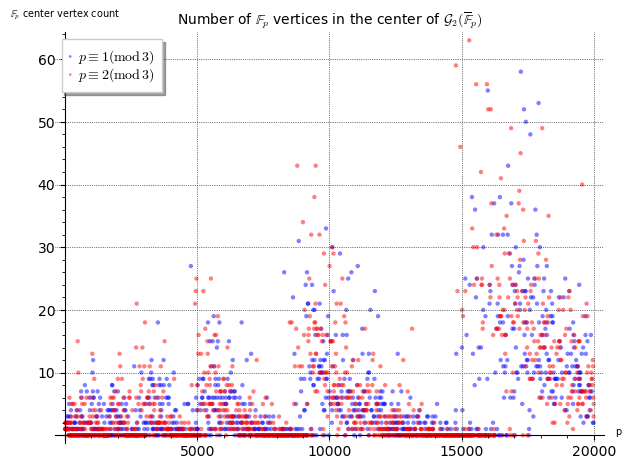}
         \caption{$p\pmod{3}$}
         \label{fig:y equals x}
     \end{subfigure}
     \hspace*{-30pt}
     \begin{subfigure}[b]{0.53\textwidth}
         \centering
         \includegraphics[width=\textwidth]{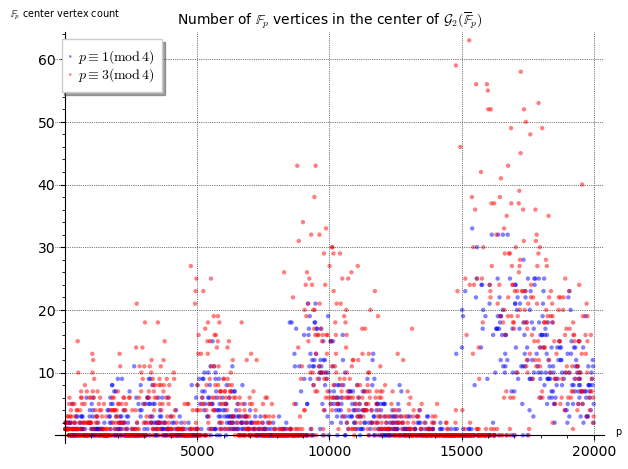}
         \caption{$p \pmod{4}$}
         \label{fig:five over x}
     \end{subfigure}
        \caption{Number of $\Fp$-vertices in the center, sorted by congruence class of~$p$. On the left, blue and red represent $p\equiv 1 \pmod{3}$ and $p \equiv 2 \pmod{3}$, respectively. On the right, blue and red correspond to $p \equiv 1 \pmod{4}$ and $p \equiv 3 \pmod{4}$, respectively.}
        \label{fig:FpCenterSorted}
\end{figure}

We also investigated the likelihood of $1728$ belonging to the center of $\mathcal{G}_2(\Fpbar)$. Overwhelmingly, this is not the case: out of the 1135 primes $p\equiv 3\pmod{4}$ with $5\leq p < 20000$, the only primes $p$ for which $1728$ lies in the center of $\mathcal{G}_2(\Fpbar)$ are $p = 7, 11, 19$.

\section{Conclusions and future work} \label{sec:concl}

The spine $\slp$ of the supersingular $\ell$-isogeny graph $\gfb$, our main protagonist herein, is obtained by mapping the supersingular $\ell$-isogeny graph $\gf$ into $\gfb$ via a natural two-step process. When passing from $\Fp$-isomorphism classes of curves to $\Fpbar$-isomorphism classes, vertices of $\gf$ representing $j$-invariants of twists are identified, leading to either stacking or folding of connected components of $\gf$. Components may be joined via attachment at a vertex (for $\ell > 2$ only, and only for $j$-invariant 1728) or an edge. Passing from $\ell$-isogenies over $\Fp$ to those over $\Fpbar$ subsequently introduces new edges. 

The authors of \cite{Arpin} provided the first major insight into this arguably surprisingly predictable process. Our structure theorems in Sections~\ref{sec:ell2structurethms}
and~\ref{sec:ell3structurethms} offer a refinement of their work by characterizing this behavior almost completely in the cases $\ell = 2, 3$ via congruence conditions on $p$, and outlining a general road map for determining $\slp$ for larger primes $\ell$. For any particular pair $(\ell, p)$, the graphs $\gf$, $\slp$ and $\gfb$ can be explicitly generated using our code at \cite{LUCANTGitHub}; for small primes $p$ and $\ell = 2, 3$, they are described explicitly in a separate document there and cited throughout this paper. 

Our structure theorems make it possible to  determine the diameter of $\slp[2]$, i.e.\ the largest distance between any two vertices of $\slp[2]$. This is entirely explicit, and shows that the diameter tends to be very small, when $p \not \equiv 7 \pmod{8}$; in the case $p \equiv 7 \pmod{8}$, the diameter is determined by the order of the ideal class represented by a prime ideal above 2 in the class group of $\mathbb{Q}(\sqrt{-p})$.

It is natural to ask how the spine is situated inside the full $\ell$-isogeny graph. To that end, we considered the centers of both $\slp[2]$ and $\mathcal{G}_2(\Fp)$, i.e.\ the collection of vertices whose furthest distance to any other vertex is minimal. We found that the count of center vertices in $\gfb$ defined over $\Fp$, as well as the size of the full center of~$\mathcal{G}(\Fpbar)$, follow a remarkable wave-like pattern as $p$ grows. A similar pattern was observed for $\ell = 3$. The community-sourced explanation shows that this pattern is evidence of supersingular elliptic curve isogeny graphs behaving as random graphs.

Our continuing exploration of graph theoretic features of $\slp$ gives us new insight into the cryptographically relevant heuristic assumptions we make about $\mathcal{G}_\ell(\Fpbar)$. Beyond the findings reported herein, we conducted extensive numerical experiments generating a substantial volume of data on both internal and external connectivity properties of $\slp$, as well as counts and proportions of vertices in the periphery of $\slp$ and $\mathcal{G}_\ell(\Fpbar)$. Our findings raise a number intriguing questions (and answers, thanks to our community); analyzing and understanding the results of our experiments is very much a work in progress as we strive to shed further light on the structural features and patters found in supersingular elliptic curve $\ell$-isogeny graphs.

\bibliographystyle{alpha}
\bibliography{EllipticCurveBib}

\newcommand{\etalchar}[1]{$^{#1}$}
\begin{thebibliography}{ACNL{\etalchar{+}}23}

\bibitem[ACNL{\etalchar{+}}23]{Arpin}
Sarah Arpin, Catalina Camacho-Navarro, Kristin Lauter, Joelle Lim, Kristina Nelson, Travis Scholl, and Jana Sot\'{a}kov\'{a}.
\newblock Adventures in {S}upersingularland.
\newblock {\em Exp. Math.}, 32(2):241--268, 2023.

\bibitem[CJS14]{Childs14}
Andrew Childs, David Jao, and Vladimir Soukharev.
\newblock Constructing elliptic curve isogenies in quantum subexponential time.
\newblock {\em J. Math. Cryptol.}, 8(1):1--29, 2014.

\bibitem[CLG09]{CGL09}
Denis~X. Charles, Kristin~E. Lauter, and Eyal~Z. Goren.
\newblock Cryptographic hash functions from expander graphs.
\newblock {\em J. Cryptology}, 22(1):93--113, 2009.

\bibitem[Cox22]{Cox}
David~A. Cox.
\newblock {\em Primes of the form {$x^2+ny^2$}---{F}ermat, class field theory, and complex multiplication}.
\newblock AMS Chelsea Publishing, Providence, RI, third edition, 2022.
\newblock With contributions by Roger Lipsett.

\bibitem[DFKL{\etalchar{+}}20]{SqiSign}
Luca De~Feo, David Kohel, Antonin Leroux, Christophe Petit, and Benjamin Wesolowski.
\newblock S{QIS}ign: compact post-quantum signatures from quaternions and isogenies.
\newblock In {\em Advances in cryptology---{ASIACRYPT} 2020. {P}art {I}}, volume 12491 of {\em Lecture Notes in Comput. Sci.}, pages 64--93. Springer, Cham, [2020] \copyright 2020.

\bibitem[DG16]{DelfGalbraith}
Christina Delfs and Steven~D. Galbraith.
\newblock Computing isogenies between supersingular elliptic curves over {$\Bbb{F}_p$}.
\newblock {\em Des. Codes Cryptogr.}, 78(2):425--440, 2016.

\bibitem[FFK{\etalchar{+}}23]{Scallop}
Luca~De Feo, Tako~Boris Fouotsa, P\'eter Kutas, Antonin Leroux, Simon-Philipp Merz, Lorenz Panny, and Benjamin Wesolowski.
\newblock S{CALLOP}: scaling the {CSI}-{F}i{S}h.
\newblock In {\em Public-key cryptography---{PKC} 2023. {P}art {I}}, volume 13940 of {\em Lecture Notes in Comput. Sci.}, pages 345--375. Springer, Cham, [2023] \copyright 2023.

\bibitem[Hed25]{LUCANTGitHub}
Taha Hedayat.
\newblock Lucant-2025-supersingular-ell-isogeny-spine.
\newblock \url{https://github.com/TahaHedayat/LUCANT-2025-Supersingular-Ell-Isogeny-Spine}, 2025.

\bibitem[JW09]{JW}
Michael~J. Jacobson, Jr. and Hugh~C. Williams.
\newblock {\em Solving the {P}ell equation}.
\newblock CMS Books in Mathematics/Ouvrages de Math\'ematiques de la SMC. Springer, New York, 2009.

\bibitem[S{\etalchar{+}}25]{sage}
William~A.\ Stein et~al.
\newblock {\em {S}age {M}athematics {S}oftware ({V}ersion 10.4)}.
\newblock The Sage Development Team, 2025.
\newblock \url{http://www.sagemath.org}.

\bibitem[Sut]{drewmod}
Andrew~V. Sutherland.
\newblock Modular polynomials.
\newblock \url{https://math.mit.edu/~drew/ClassicalModPolys.html}.
\newblock Accessed: 2025-01-17.

\end{thebibliography}
\end{document}